\documentclass[11pt]{amsart}

\usepackage{amsmath, amssymb, amsthm, amscd}
\usepackage{float}
\usepackage{graphicx, epstopdf}
\usepackage{fullpage}
\usepackage{todonotes}

\newtheorem{definition}{Definition}[section]
\newtheorem{theorem}[definition]{Theorem}

\newtheorem{lemma}[definition]{Lemma}
\newtheorem{corollary}[definition]{Corollary}
\newtheorem{proposition}[definition]{Proposition}
\newtheorem{remark}[definition]{Remark}

\newcommand{\Q}{\mathbf{Q}}
\newcommand{\Z}{\mathbf{Z}}
\newcommand{\R}{\mathbf{R}}
\newcommand{\C}{\mathbf{C}}
\newcommand{\bs}{\backslash}

\renewcommand{\Im}{\textrm{Im}}

\DeclareMathOperator{\SL}{SL}
\DeclareMathOperator{\GL}{GL}
\DeclareMathOperator{\Sp}{Sp}
\DeclareMathOperator{\rk}{rk}
\DeclareMathOperator{\tr}{tr}
\DeclareMathOperator{\GSp}{GSp}

\DeclareMathOperator{\Gr}{Gr}
\DeclareMathOperator{\Stab}{Stab}

\DeclareMathOperator{\vol}{vol}

\title{Hecke eigenvalues of Klingen--Eisenstein series of squarefree level}
\author{Martin J. Dickson}
\date{\today}

\begin{document}
\maketitle

\begin{abstract}We compute the intertwining relation between the Hecke operators and the Siegel lowering operators on Siegel modular forms of arbitrary level $N$ and character $\chi$ by using formulas for the action of the Hecke operators on Fourier expansions.  Using an explicit description of the Satake compactification of $\Gamma_0^{(n)}(N) \backslash \mathfrak{H}_n$ when $N$ is squarefree we extend this to give intertwining relations for each cusp.  As an application we give formulas for the action of Hecke operators on the space of Klingen--Eisenstein series of squarefree level $N$, for primes $p \nmid N$.\end{abstract}

\section{Introduction}

Let $\mathcal{M}_k^{(n)}(N, \chi)$ denote the space of Siegel modular forms of degree $n$, weight $k$, level $N$, and character $\chi$ modulo $N$.  Since the theta series encoding the number of times a fixed quadratic form in $2k$ variables represents a quadratic form in $n$ variables defines an element of such a space, it is natural from the viewpoint of the arithmetic of quadratic forms to want to understand $\mathcal{M}_k^{(n)}(N, \chi)$ as explicitly as possible.  The first thing one should exploit in this endeavour is the fact that the vector space $\mathcal{M}_k^{(n)}(N, \chi)$ can be decomposed as the direct sum of the space $\mathcal{S}_k^{(n)}(N, \chi)$ of cusp forms and the complementary space of Eisenstein series.  The space of Eisenstein series itself decomposes further, as a direct sum of subspaces indexed by integers $0 \leq r < n$, where each subspace consists of the Klingen--Eisenstein series formed from Siegel cusp forms of degree $r$.  The basic philosophy is that one should understand an Eisenstein series of degree $n$ just as well as one understands the cusp form of degree $r$ it was lifted from.  In this paper we will describe a method for making this practicable, when $N$ is square-free.\\

We focus our attention on the action of Hecke operators on Klingen--Eisenstein series, although we hope that parts of the set-up we describe will be useful in examining other features of $\mathcal{M}_k^{(n)}(N, \chi)$.  In order to understand the structure of the Eisenstein part of $\mathcal{M}_k^{(n)}(N, \chi)$ as a Hecke module, the first step is to derive a relation between the action of Hecke operators on modular forms of degree $n$ and modular forms of degree $n-1$.  For $F \in \mathcal{M}_k^{(n)}(1) := \mathcal{M}_k(1, \mathbf{1})$ it is not difficult to show\begin{footnote}{See for example \cite{Freitag1983} Satz IV.4.4, but beware the differences in normalisation.  Our normalisation of the Hecke operators is introduced in \eqref{eqn:hecke-cosets} and \eqref{eqn:hecke-op-normalization}.}\end{footnote} that 
\[\Phi(F | T^{(n)}(p)) = (1+p^{k-n})\Phi(F) | T^{(n-1)}(p).\]
Slightly more complicated, but still completely explicit, relations were found for the remaining Hecke operators $T_j(p^2)$ acting on $\mathcal{M}_k^{(n)}(1)$ in \cite{Krieg1986}.  As noted in \cite{Krieg1986} there is also a version of intertwining relationship due to \v{Z}arkovskaja (\cite{Zarkovskaja1974}) which holds in more generality; however this not explicit enough for our purposes.  The first main result of this paper is a completely explicit forms of the intertwining relations for arbitrary level and character\begin{footnote}{Which satisfies the natural condition explained in Remark \ref{rmk:character-parity}.}\end{footnote}:

\begin{theorem}\label{thm:first-relation}  Let $n$, $k$, and $N$ be positive integers, let $\chi$ be a character modulo $N$ such that $\chi(-1) = (-1)^{k}$, let $F \in \mathcal{M}_k^{(n)}(N, \chi)$, and let $p$ be any prime.  Then
\[\begin{aligned} \Phi(F|T^{(n)}(p, \chi)) &= c^{(n-1)}(\chi) \Phi(F)|T^{(n-1)}(p, \chi), \\
\Phi(F | T_j^{(n)}(p^2, \chi)) &= c_{j, j}^{(n-1)}(\chi)\Phi(F) | T^{(n-1)}_j(p^2, \chi) \\
&\qquad+ c_{j, j-1}^{(n-1)}(\chi)\Phi(F) | T^{(n-1)}_{j-1}(p^2, \chi) \\
&\qquad+ c_{j, j-2}^{(n-1)}(\chi)\Phi(F)| T^{(n-1)}_{j-2}(p^2, \chi) \end{aligned}\]
where
\[\begin{aligned}
c^{(n-1)}(\chi) &= (1 + \chi(p) p^{k-n}), \\
c_{j, j}^{(n-1)}(\chi) &= \chi(p)p^{j+k-2n}, \\
c_{j, j-1}^{(n-1)}(\chi) &= \chi(p^2)p^{2k-2n} + \chi(p)(p^{j+k-2n} - p^{j+k-2n-1}) + 1, \\
c_{j, j-2}^{(n-1)}(\chi) &=  \chi(p)(p^{k-j+1} - p^{k+j-2n-1}),\end{aligned}\]
with the understanding that $T_j^{(n-1)}(p, \chi)$ is the zero operator for $j \in \{-2, -1, n\}$.
  
\end{theorem}

When $N=1$ this reduces to the main result of \cite{Krieg1986} (after accounting for the differences in normalisation).  However, our method of proof, which also applies for bad primes, is quite different, and is based on the action of Hecke operators on Fourier expansions.  The same style of argument works in all cases: it is straightforward for $T^{(n)}(p, \chi)$ but far more involved for $T_j^{(n)}(p^2, \chi)$.  We therefore provide full details in the latter case in \S\ref{sctn:first-relations} and some indications of how one can argue similarly for the former in \S\ref{sctn:easy-relations}.  From the definitions in \S\ref{sctn:preliminaries} we see that there is nothing to prove for $j=0$; we will deduce the relations for $T_j^{(n)}(p^2, \chi)$ when $j>0$ from analogous relations for a set of averaged operators $\widetilde{T}_j^{(n)}(p^2, \chi)$.\\

Of course Theorem \ref{thm:first-relation} only refers to the output at a single cusp, but we should really be examining the behaviour of $F$ at all $(n-1)$-cusps simultaneously.  It is therefore necessary to consider the question of intertwining between the action of Hecke operators and restrictions to \textit{other} cusps.  In this consideration we restrict to the case when $N$ is squarefree.  We begin by providing a description of the Satake compactification $\Gamma_0^{(n)}(N) \backslash \mathfrak{H}_n^*$ of $\Gamma_0^{(n)}(N) \backslash \mathfrak{H}_n$ when $N$ is squarefree.  The compactification is obtained by adding quotients of $\mathfrak{H}_r$ to the boundary.  We describe these in detail, and how they intersect each other in lower dimensional components; see Theorem \ref{thm:description-of-boundary} for a precise statement.  Theorem \ref{thm:description-of-boundary} gives more information than is strictly necessary for our applications, but the extra information is easily obtained and perhaps of independent interest.  Using this description, we may parameterise the $r$-cusps of $\Gamma_0^{(n)}(N) \backslash \mathfrak{H}_n^*$ with sequence $(l_{n-r},...,l_1)$ of divisors of $N$ which are pairwise coprime.\begin{footnote}{The outr\'{e} labelling of the indices is explained by the discussion in \S\ref{sctn:satake-compactification-description}.}\end{footnote}  Given such a sequence, we define $l_0 = N/l_{n-r}\hdots l_1$.  In particular, an $(n-1)$-cusp corresponds to a divisor $l_1$ of $N$; we write $\Phi_{l_1}$ for the map restricting to that cusp\begin{footnote}{This depends on a choice of coset representative (c.f. \eqref{eqn:lowering-dependency-on-rep}), see \S\ref{sctn:intertwining-relations-any-cusp} for our precise definition.}\end{footnote}; with our definitions, $\Phi_1$ will be the usual lowering operator $\Phi$.  Since we have restricted to $N$ squarefree, we can represent cusps by Atkin--Lehner style operators; thus, using an argument similar to one used in \cite{Asai1976} for modular forms of degree $1$, we obtain relations which differ to those of Theorem \ref{thm:first-relation} only in the characters:

\begin{theorem}\label{thm:any-cusp-relation}  Let $n$ and $k$ be positive integers, let $N$ be a squarefree positive integer, let $\chi$ be a character modulo $N$ such that $\chi(-1) = (-1)^{k}$, let $p \nmid N$ be prime, and let $F \in \mathcal{M}_k^{(n)}(N, \chi)$.  Then
\[\begin{aligned} \Phi_{l_1}(F|T^{(n)}(p, \chi)) &= \chi_{l_1}(p^n)c^{(n-1)}(\overline{\chi}_{l_1} \chi_{l_0}) \Phi_{l_1}(F)|T^{(n-1)}(p, \overline{\chi}_{l_1}\chi_{l_0}), \\
\Phi_{l_1}(F | T_j^{(n)}(p^2, \chi)) &= \chi_{l_1}(p^{2n}) \left[c_{j, j}^{(n-1)}(\overline{\chi}_{l_1}\chi_{l_0})\Phi_{l_1}(F) | T^{(n-1)}_j(p^2, \overline{\chi}_{l_1}\chi_{l_0}) \right.\\
&\qquad\left.+ c_{j, j-1}^{(n-1)}(\overline{\chi}_{l_1}\chi_{l_0})\Phi_{l_1}(F) | T^{(n-1)}_{j-1}(p^2, \overline{\chi}_{l_1}\chi_{l_0}) \right.\\
&\qquad\left.+ c_{j, j-2}^{(n-1)}(\overline{\chi}_{l_1}\chi_{l_0})\Phi_{l_1}(F)| T^{(n-1)}_{j-2}(p^2, \overline{\chi}_{l_1}\chi_{l_0})\right], \end{aligned}\]
where $c^{(n-1)}$, $c^{(n-1)}_{j, j}$, $c^{(n-1)}_{j, j-1}$ and $c^{(n-1)}_{j, j-2}$ are as in Theorem \ref{thm:first-relation}, and we adopt the same convention that $T_j^{(n-1)}$ is the zero operator for $j \in\{n, -1, -2\}$.\end{theorem}

With Theorem \ref{thm:any-cusp-relation} in place we then proceed to the main goal of this paper, which is to describe the action of the Hecke operators on the full space of Eisenstein series.  We continue to work with $N$ squarefree and $p \nmid N$ prime.  Since we are working with ``good'' Hecke operators it is not difficult to show, using the normality of these Hecke operators with respect to the Petersson inner product, that the Klingen lift of a degree $r$ cuspidal eigenform to a degree $n$ modular form is again an eigenform.  Note however that the definition of both the Siegel lowering operator and the Klingen lift depend on the choice of coset representative.  In many situations, for example defining cuspidality, the exact definition is ultimately not important; however, to understand these operations in the presence of Hecke operators requires some care and consistency.  In \S\ref{sctn:intertwining-relations-any-cusp} and \S\ref{action-of-hecke-on-eis} we clarify the dependency of the Klingen lift and Siegel lowering operator on ancillary choices of coset representatives and provide definitions motivated by the results in \S\ref{sctn:satake-compactification-description} regarding the boundary of $\Gamma_0^{(n)}(N) \backslash \mathfrak{H}_N^*$.\\  

With this theory in place it is then a simple matter to describe the action of the Hecke operators.  Since the results of \S\ref{sctn:satake-compactification-description} show that each boundary component of $\Gamma_0^{(n)}(N) \backslash \mathfrak{H}_n^*$ is itself of the form $\Gamma_0^{(r)}(N) \backslash \mathfrak{H}_r$ for some $0 \leq r < n$, we are able to work iteratively; keeping track of the action of the Hecke operators at each stage we are eventually able to provide formulas for the degree $n$ Hecke eigenvalues in terms of the Hecke eigenvalues of the degree $r$ cusp form.  These formulas are specific to the $r$-cusp which we lift from.  Lifting a basis of cuspidal eigenforms from all $r$-cusps, for all $0 \leq r < n$, the lifting process provides a basis of eigenforms for the space of Eisenstein series.  The main result of this paper is a formula for the Hecke eigenvalues of this basis.\\

To illustrate the point, let us now state the result a simple, illustrative case, namely the Eisenstein series of degree two associated to a cusp form of degree one, so $n=2$ and $r=1$ in the above paragraph.  Let $N$ be squarefree, $l_1$ a divisor of $N$ corresponding to a $1$-cusp on $\Gamma_0^{(2)}(N) \backslash \mathfrak{H}_2^*$, set $l_0 = N/l_1$, let $F \in \mathcal{S}_k^{(1)}(N, \overline{\chi}_{l_1} \chi_{l_0})$ be a cusp form of degree one on the $1$-cusp corresponding to $l_1$, and write $E_{l_1}(F) \in \mathcal{M}_k^{(2)}(N, \chi)$ for the Klingen lift of $F$.\begin{footnote}{We require $F$ to have character $\overline{\chi}_{l_1} \chi_{l_0}$ so that $E_{l_1}(F)$ has character $\chi$ (c.f. the definition of $E_{l_1}$ in \S\ref{action-of-hecke-on-eis}).}\end{footnote}  We assume that $F$ is an eigenfunction of the usual Hecke operator $T^{(1)}(p, \overline{\chi}_{l_1} \chi_{l_0})$, say with eigenvalue $\lambda^{(1)}(p, \overline{\chi}_{l_1} \chi_{l_0})$.  We also write $\lambda_0^{(1)}(p^2, \overline{\chi}_{l_1} \chi_{l_0}) = p^{k-3} \overline{\chi}_{l_1} \chi_{l_0}(p)$ and $\lambda_1^{(1)}(p^2, \overline{\chi}_{l_1} \chi_{l_0}) = \lambda^{(1)}(p, \overline{\chi}_{l_1} \chi_{l_0})^2 - (1 + \chi_{l_1} \overline{\chi}_{l_0}(p) p^{k-1})\lambda_0^{(1)}(p^2, \overline{\chi}_{l_1} \chi_{l_0})$.\begin{footnote}{Here $\lambda_0^{(1)}(p^2, \overline{\chi}_{l_1} \chi_{l_0})$ is the eigenvalue of (any) $F \in \mathcal{S}_k^{(1)}(N, \overline{\chi}_{l_1} \chi_{l_0})$ under $T_0^{(1)}(p^2, \overline{\chi}_{l_1} \chi_{l_0})$, which is an element of the Hecke algebra that acts as this scalar.  Similarly, $\lambda_1^{(1)}(p^2, \overline{\chi}_{l_1} \chi_{l_0})$ is the eigenvalue of $F$ under $T_1^{(1)}(p^2, \overline{\chi}_{l_1} \chi_{l_0})$; the relation $T_1^{(1)}(p^2, \overline{\chi}_{l_1} \chi_{l_0}) = T^{(1)}(p, \overline{\chi}_{l_1} \chi_{l_0})^2 - (1 + \chi_{l_1} \overline{\chi}_{l_0}(p) p^{k-1})T_0^{(1)}(p^2, \overline{\chi}_{l_1} \chi_{l_0})$ holds in the Hecke algebra.}\end{footnote}  For $n=2$ it suffices to consider the operators $T^{(2)}(p, \chi)$ and $T_1^{(2)}(p^2, \chi)$, since $T_0^{(2)}(p^2, \chi)$ acts as a scalar and $T_2^{(2)}(p^2, \chi)$ is algebraically dependent on the others.  For the interesting operators we have:

\begin{theorem}  Continue with the notation of the previous paragraph, and let $p \nmid N$ be prime.  Then
\[\begin{aligned}E_{l_1}(F) | T^{(2)}(p, \chi) &= \chi_{l_1}(p^2) c^{(1)}(\overline{\chi}_{l_1} \chi_{l_0}) \lambda^{(1)}(p, \overline{\chi}_{l_1} \chi_{l_0}) E_{l_1}(F), \\
E_{l_1}(F) | T_1^{(2)}(p^2, \chi)&= \chi_{l_1}(p^4)[c_{1, 1}^{(1)}(\overline{\chi}_{l_1} \chi_{l_0})\lambda_1^{(1)}(p^2, \overline{\chi}_{l_1} \chi_{l_0}) + c_{1, 0}^{(1)}(\overline{\chi}_{l_1} \chi_{l_0}) \lambda_0^{(1)}(p^2, \overline{\chi}_{l_1} \chi_{l_0})]E_{l_1}(F),\end{aligned}\]
with $c^{(1)}$, $c_{1, 1}^{(1)}$, and $c_{1, 0}^{(1)}$ as in Theorem \ref{thm:first-relation}.\end{theorem}

This is a special case of the main result, Theorem \ref{thm:action-of-hecke-on-klingen}, which gives recursive formulas in any degree $n$ (still assuming $N$ is squarefree and $p \nmid N$).  In the statement of Theorem \ref{thm:action-of-hecke-on-klingen} we opt to leave the formulas recursive, since the solution of the recursion does not seem to be particularly illuminating.\\

Specializing to the case of Siegel--Eisenstein series, i.e. elements of $\mathcal{M}_k^{(n)}(N, \chi)$ lifted from $0$-cusps, Theorem \ref{thm:action-of-hecke-on-klingen} gives another proof of the results of \cite{Walling2012} and \cite{Walling2014} in the case when $N$ is squarefree and $p \nmid N$.  It is not clear how the methods of this paper should extend to the case of $N$ no longer squarefree, since whilst it is still possible to explicitly describe the boundary of $\Gamma_0^{(n)}(N) \backslash \mathfrak{H}_n^*$, it is no longer the case that we can reach each cusp with Atkin--Lehner style operators.  We use this latter fact crucially in the case when $N$ is squarefree, and experience from the degree $n=1$ case indicates that the relationship between Hecke operators at cusps which are \textit{not} related by Atkin--Lehner style operators will be much less transparent.  On the other hand, we expect that the methods of this paper can be extended to the case $p \mid N$ when $N$ is squarefree.  Note that the Hecke operators when $p \mid N$ are no longer normal, so the Klingen lifts may no longer be eigenfunctions of these bad Hecke operators.  In fact the non-diagonality of the action of the bad Hecke operators on Siegel--Eisenstein series of degree two, squarefree level, and trivial character described in \cite{Walling2012} gives enough linear relations to allow one to deduce formulas for the Fourier coefficients of a full basis for the space of level $N$ Siegel--Eisenstein series from well-known formulas in level one, as explained in \cite{Dickson2014b}.  It would be interesting to investigate this possibility more generally in the case of Klingen--Eisenstein series, and compare the results to the formulas of \cite{Boecherer1982} for the Fourier coefficients of degree two Klingen--Eisenstein series.\\

\textbf{Acknowledgements.}  This work formed part of the author's PhD thesis, and he would like to thank his supervisor Lynne Walling for her guidance.  He would also like to thank his examiners Tim Dokchitser and Nils Skoruppa for their comments.

\section{Preliminaries}\label{sctn:preliminaries}

For $n \in \Z_{\geq 1}$ the algebraic group $\GSp_{2n}$ is defined as
\[\GSp_{2n} = \{g \in \GL_{2n};\: {}^tg J g = \mu_n(g) J\text{ for some }\mu_n(g) \in \GL_1\},\]
where 
\[J = \begin{pmatrix} 0_n & -1_n \\ 1_n & 0_n\end{pmatrix}.\]
The map $\mu_n : \GSp_{2n} \to \GL_1$ is a homomorphism, we define $\Sp_{2n}$ as its kernel.  If $R$ is a subring of $\R$ we write $\GSp_{2n}^+(R)$ for the subgroup of $\GSp_{2n}(R)$ consisting of those $g$ with $\mu_n(g) > 0$.  Let
\[\mathfrak{H}_n = \{Z \in \C^{n \times n}_{\text{sym}};\: \Im(Z)>0\}\]
be Siegel's upper half space of degree $n$.  Then $\GSp_{2n}^+(\R)$ acts on $\mathfrak{H}_n$ by
\begin{equation}\label{eqn:gsp4-action-on-hn}(\gamma, Z) \mapsto \gamma \langle Z \rangle = (AZ + B)(CZ+D)^{-1},\end{equation}
where $\gamma = \left(\begin{smallmatrix} A & B \\ C & D \end{smallmatrix}\right) \in \GSp_{2n}^+(\R)$.  For $k$ a positive integer we also define an action of $\GSp_{2n}^+(\R)$ on functions $F : \mathfrak{H}_n \to \C$ by
\[(F|_k \gamma)(Z) = \mu_n(\gamma)^{nk/2} j(\gamma, Z)^{-k} F(\gamma \langle Z \rangle),\]
where $j(\gamma, Z) = \det(CZ + D)$.  We will often simply write $F|\gamma$ in place of $F|_k \gamma$, since $k$ should be clear from the context.  It will be useful for us to have an interpretation for this formula even when $n=0$.  In this, $\mathfrak{H}_0$ becomes a point, the point denoted by $\infty$; a function $F : \mathfrak{H}_0 \to \C$ is therefore constant, and we identify $F$ with the value it takes.  Finally, for $n=0$, any $k$, and any $\gamma$, the action $F|_k \gamma$ is taken to be trivial.\\

We work with modular forms on the congruence subgroup
\[\Gamma_0^{(n)}(N) = \left\{\begin{pmatrix} A & B \\ C & D\end{pmatrix} \in \Sp_{2n}(\Z);\: C \equiv 0 \bmod N\right\}.\]
Given a Dirichlet character $\chi$ modulo $N$, we define a character of $\Gamma_0^{(n)}(N)$, also denoted $\chi$, by $\chi\left(\left(\begin{smallmatrix} A & B \\ C & D \end{smallmatrix}\right)\right) = \chi(\det(D)) (= \overline{\chi}(\det(A)))$.  Given $n \in \Z_{\geq 2}$, $k, N \in \Z_{\geq 1}$ and $\chi$ a Dirichlet character modulo $N$, we define
\[\mathcal{M}_k^{(n)}(N, \chi) = \{F:\mathfrak{H}_n \to \C;\:F\text{ is holomorphic};\: F|_k \gamma = \chi(\gamma) F\text{ for all }\gamma \in \Gamma_0^{(n)}(N)\}.\]
When $n=1$ we use the same definition, except that it is now necessary to additionally impose that $F$ be regular at the cusps.  The space $\mathcal{M}_k^{(n)}(N, \chi)$ is finite dimensional, and is equipped with a partially defined inner product
\begin{equation}\label{eqn:petersson-def}\langle F, G \rangle := \frac{1}{\vol(\Gamma_0^{(n)}(N) \backslash \mathfrak{H}_n)} \int_{\Gamma_0^{(n)}(N) \backslash \mathfrak{H}_n} F(Z) \overline{G(Z)} \det(Y)^{k} d\mu(Z),\end{equation}
where $Z=X+iY$ is the decomposition in to real and imaginary parts, and $d\mu(Z) = dX dY/\det(Y)^{n+1}$ is (a fixed normalization of) the $\Sp_{2n}(\R)$-invariant measure on $\mathfrak{H}_n$.\\

The Siegel lowering operator $\Phi$ is defined, for $F \in \mathcal{M}_k^{(n)}(N, \chi)$ and $Z' \in \mathfrak{H}_{n-1}$, by
\[\Phi(F)(Z') = \lim_{\lambda \rightarrow \infty} F\left(\begin{pmatrix} Z' & 0 \\ 0 & i \lambda \end{pmatrix}\right).\]
For $\gamma \in \Sp_{2n}(\Q)$, define $\Phi_{\gamma}(F) = \Phi(F|\gamma).$  The space of cusp forms is defined as
\[\mathcal{S}_k^{(n)}(N, \chi) = \{F \in \mathcal{M}_k^{(n)}(N, \chi);\: \Phi_{\gamma}(F) = 0\text{ for all }\gamma \in \Sp_{2n}(\Q)\}.\]
Note that the condition imposing cuspidality is equivalent to the (finite) condition where one replaces all $\gamma \in \Sp_{2n}(\Q)$ with a system of representatives for $\Gamma_0^{(n)}(N) \backslash \Sp_{2n}(\Q) / P_{n, n-1}(\Q)$, where for $0 \leq r \leq n$ positive integers $P_{n, r}$ is the parabolic subgroup
\begin{equation}\label{eqn:Pnr-definition}P_{n, r} = \left\{\begin{pmatrix} A_{11} & 0 & B_{11} & B_{12} \\ A_{21} & A_{22} & B_{21} & B_{22} \\ C_{11} & 0 & D_{11} & D_{12} \\ 0 & 0 & 0 & D_{22}\end{pmatrix};\: *_{11}\text{ size }r;\: *_{22}\text{ size }(n-r)\right\}.\end{equation}
There is a surjection $\omega_{n, r} : P_{n, r} \to \Sp_{2r}$, given by
\begin{equation}\label{eqn:omeganr-definition}\omega_{n, r}\left(\begin{pmatrix} A_{11} & 0 & B_{11} & B_{12} \\ A_{21} & A_{22} & B_{21} & B_{22} \\ C_{11} & 0 & D_{11} & D_{12} \\ 0 & 0 & 0 & D_{22}\end{pmatrix}\right) = \begin{pmatrix} A_{11} & B_{11} \\ C_{11} & D_{11} \end{pmatrix}.\end{equation}
Note that $\Phi_{\gamma}$ is not well-defined on the double coset $\Gamma_0^{(n)}(N) \gamma P_{n, r}(\Q)$.  Indeed, if $\gamma' \in \Gamma_0^{(n)}(N)$ and $\delta \in P_{n, n-1}(\Q)$ is written in the form \eqref{eqn:Pnr-definition}, then
\begin{equation}\label{eqn:lowering-dependency-on-rep}\Phi_{\gamma' \gamma \delta}(F) = \chi(\gamma')D_{22}^{-k} \Phi_\gamma(F) | \omega_{n, n-1}(\delta).\end{equation}
Note that $D_{22} \in \Q^{\times}$, since $\delta \in P_{n, n-1}(\Q)$, so the choice of representative does not matter for defining cuspidality.  However, more care is required for other applications.

\begin{remark}\label{rmk:character-parity}  Note, since $-1_{2n} \in \Gamma_0^{(n)}(N)$, that if $\chi((-1)^n) \neq (-1)^{nk}$ then $\mathcal{M}_k^{(n)}(N, \chi) = 0$.  In \S\ref{sctn:satake-compactification-description} we will choose representatives and define lowering operators
\[\Phi_l : \mathcal{M}_k^{(n)}(N, \chi) \to \mathcal{M}_k^{(n-1)}(N, \overline{\chi_l} \chi_{N/l}),\]
the cuspidality condition being $\Phi_l(F) \equiv 0$ for each $l$.  The above vanishing condition then applies to the target of this map, so if $\chi((-1)^{n-1}) \neq (-1)^{(n-1)k}$ then $\Phi_l(F)$ must be zero.  Thus if $\mathcal{M}_k^{(n)}(N, \chi)$ is to contain non-cusp forms then we require $\chi(-1) = (-1)^k$, and we therefore make this natural assumption.  Note that if $\chi(-1) \neq (-1)^k$ then $\mathcal{M}_k^{(n)}(N, \chi)$ may contain cusp forms, for example $\mathcal{S}_{35}^{(2)}(1)$ is non-zero. \end{remark}

Let $F \in \mathcal{M}_k^{(n)}(N, \chi)$, so it has a Fourier expansion
\[F(Z) = \sum_{T\geq 0} a(T; F) e(\tr(TZ)),\]
where $T$ varies over all positive semi-definite matrices symmetric matrices of size $n$ which are semi-integral (i.e. $T = (t_{ij})$ with $t_{ij} \in \frac{1}{2}\Z, t_{ii} \in \Z$), and $e(z) = e^{2 \pi i z}$ for $z \in \C$.  It will be convenient for us to introduce another indexing set for the Fourier expansion.  Let $\Lambda$ be an even lattice, i.e. a lattice equipped with a $\Z$-valued quadratic form.  Attached to $\Lambda$ we have a collection of even integral (i.e. $T = (t_{ij})$ with $t_{ij} \in \Z$, $t_{ii} \in 2\Z$) Gram matrices $\{{}^tGTG;\: G \in \GL_n(\Z)\}$.  Since we assume that $\chi(-1) = (-1)^k$, the modularity of $F$ then implies $a(\frac{1}{2}T; F) = a(\frac{1}{2}{}^tGTG; F)$.\begin{footnote}{To circumvent the assumption $\chi(-1) = (-1)^k$ one may work with oriented lattices, but since we are interested in Eisenstein series the point is moot.}\end{footnote}  It then make sense to define $a(\Lambda; F) = a(\frac{1}{2}T; F)$ where $T$ is the Gram matrix for any basis of $\Lambda$.  Now varying $\Lambda$ over all even lattices we obtain all possible (classes of) $T$, so allowing $\Lambda$ to vary thus in the following sum we have
\begin{equation}\label{eqn:lattice-fourier-expansion}F(Z) = \sum_{\Lambda} a(\Lambda; F) e\{\Lambda Z\}.\end{equation}
Here 
\[e\{\Lambda Z\} = \sum_{G \in O(\Lambda) \backslash \GL_n(\Z)} e\left(\frac{1}{2}\tr({}^tG T G Z)\right),\]
where $O(\Lambda)$ is the orthogonal group of the lattice $\Lambda$.  If we refer to $a(\Lambda; F)$ when the quadratic form on $\Lambda$ is not integral then we understand $a(\Lambda; F) = 0$.\\

We now introduce the Hecke operators.  Let
\[\Delta_0^{(n)}(N) = \left\{\delta = \begin{pmatrix} A & B \\ C & D \end{pmatrix} \in \GSp_{2n}^+(\Q) \cap \Z^{2n \times 2n};\: C \equiv 0 \bmod N;\: \gcd(\det(A), N) = 1\right\}.\]
The character $\chi$ extends to a character of $\Delta_0^{(n)}(N)$ by $\chi(\delta) = \overline{\chi}(\det(A))$.  We write $\mathcal{H}^{(n)}(N)$ for the Hecke algebra of the pair $(\Gamma_0^{(n)}(N), \Delta_0^{(n)}(N))$.  Focussing on a prime $p$ we define the local Hecke algebra $\mathcal{H}_p^{(n)}(N)$ to be the ring of $\Z$-linear combinations of double coset $\Gamma_0^{(n)}(N) M \Gamma_0^{(n)}(N)$, where
\[M = \{g \in \Delta_0^{(n)}(N);\: \mu(g)\text{ is a power of }p\}.\] 
Define the double cosets 
\begin{equation}\label{eqn:hecke-cosets} \begin{aligned} T^{(n)}(p) &:= \Gamma_0^{(n)}(N) \begin{pmatrix} 1_n & \\ & p1_n \end{pmatrix} \Gamma_0^{(n)}(N) \\
T_j^{(n)}(p^2) &:= \Gamma_0^{(n)}(N) \begin{pmatrix} 1_j & & & \\ & p1_{n-j} & & \\ & & p^2 1_j & \\ & & & p1_{n-j}\end{pmatrix} \Gamma_0^{(n)}(N), \text{ for }0 \leq j \leq n;\end{aligned} \end{equation}
it is well-known that $\{T^{(n)}(p)\} \cup \{T_j^{(n)}(p);\: 0 \leq j \leq n\}$ generates the algebra $\mathcal{H}_p^{(n)}(N)$.  We let an element $\Gamma_0^{(n)}(N) \alpha \Gamma_0^{(n)}(N) \in \mathcal{H}_p^{(n)}(N)$ act on $\mathcal{M}_k^{(n)}(N, \chi)$ by writing
\[\Gamma_0^{(n)}(N) \alpha \Gamma_0^{(n)}(N) = \bigsqcup_v \Gamma_0^{(n)}(N) \alpha_v\]
and defining
\begin{equation}\label{eqn:hecke-op-normalization}F | \Gamma_0^{(n)}(N) \alpha \Gamma_0^{(n)}(N) = \mu(\alpha)^{\frac{nk}{2} - \frac{n(n+1)}{2}} \sum_v \overline{\chi}(\alpha_v) F| \alpha_v.\end{equation}
This is extended by linearity to an action of $\mathcal{H}_p^{(n)}(N) \otimes_\Z \C$.  We write $T^{(n)}(p, \chi)$ and $T_j^{(n)}(p^2, \chi)$ for the operators on $\mathcal{M}_k^{(n)}(N, \chi)$ defined by \eqref{eqn:hecke-cosets}.

\begin{remark}  Note that our normalisation of the slash operator differs from the classical (Andrianov) notation since we include the factor of $\mu_n(\gamma)^{nk/2}$ in order to force scalar matrices to act trivially.  However, this effect is compensated in our normalisation of the Hecke operators; the result is that our Hecke operators are the same as those in the Andrianov notation, except that we have interchanged the role of $j$ and $n-j$ in $T_j^{(n)}(p^2, \chi)$.  When $n=1$, both our slash operators and our Hecke operators are normalized as in \cite{Miyake2006}.  Note that \cite{HafnerWalling2002} uses a definition of Hecke operators that is equivalent to our double coset definition except that the representative matrices differ by a factor of $p$.  This makes no difference because we both normalize the slash operator so that scalar matrices act trivially.\end{remark}

The formula for the action of Hecke operators on Fourier expansions was found in \cite{HafnerWalling2002}.  It is most conveniently stated using the indexing of Fourier coefficients by lattices $\Lambda$ as above, and moreover it is easiest to work not with the operators $T_j(p^2)$ but rather with certain averaged versions, which we now introduce.  We will use these operators extensively in \S\ref{sctn:first-relations}, but they will not appear elsewhere in the paper.  To define them, fist let ${n \choose r}_p$ be the Gaussian binomial coefficient, i.e.
\[{n \choose r}_p = \prod_{i=1}^r \frac{p^{n-i+1}-1}{p^{r-i+1}-1}.\]
Then, for $F \in \mathcal{M}_k^{(n)}(N, \chi)$, 
\begin{equation}\label{eqn:tilde-T-definition} F | \widetilde{T}^{(n)}_j(p^2, \chi) := p^{(n-j)(n-k+1)} \overline{\chi}(p^{n-j})\sum_{t=0}^j {n-t \choose j-t}_p F | T_t^{(n)}(p^2, \chi).\end{equation}
In order to state the action of these operators on Fourier expansions we first introduce some useful notation:

\begin{definition}\label{dfn:p-type}  Let $\Lambda$ be a lattice, and $p$ be a prime.  Let $\Omega$ be a lattice such that $p\Lambda \subset \Omega \subset \Lambda$.  By the invariant factor theorem we can write
\[\begin{aligned} \Lambda &= \Lambda_0 \oplus \Lambda_1, \\
\Omega &= \Lambda_0 \oplus p\Lambda_1. \end{aligned} \] 
We call the tuple $(\rk(\Lambda_0), \rk(\Lambda_1))$ the $p$-type of $\Omega$ (in $\Lambda$).  Similarly, let $\Omega$ be a lattice such that $p \Lambda \subset \Omega \subset p^{-1} \Lambda$.  By the invariant factor theorem we can write
\[\begin{aligned} \Lambda &= \Lambda_0 \oplus \Lambda_1 \oplus \Lambda_2, \\
\Omega &= p^{-1} \Lambda_0 \oplus \Lambda_1 \oplus p\Lambda_2. \end{aligned}\]
We (again) call the tuple $(\rk(\Lambda_0), \rk(\Lambda_1), \rk(\Lambda_2))$ the $p$-type of $\Omega$ (in $\Lambda$).  \end{definition}

\begin{theorem}[Hafner--Walling, \cite{HafnerWalling2002}]\label{action-of-Tp-on-fourier-exp}  Let $F \in \mathcal{M}_k^{(n)}(N, \chi)$ have Fourier expansion (\ref{eqn:lattice-fourier-expansion}), let $p$ be any prime, and write 
\[(F|T^{(n)}(p, \chi))(Z) = \sum_\Lambda a(\Lambda; F|T^{(n)}(p, \chi)) e\{\Lambda Z\}.\]  
Then
\[a(\Lambda; F|T^{(n)}(p, \chi)) = \sum_{p\Lambda \subset \Omega \subset \Lambda} A(\Omega, \Lambda; F|T^{(n)}(p, \chi))\]
with $A(\Omega, \Lambda; F|T^{(n)}(p, \chi))$ defined as follows: let $(m_0, m_1)$ be the $p$-type of $\Omega$ in $\Lambda$, and set
\[E^{(n)}(\Omega, \Lambda) = m_0k + \frac{m_1(m_1+1)}{2} - \frac{n(n+1)}{2};\]
and if $\Omega$ has quadratic form $Q$ let $\Omega^{1/p}$ denote the same lattice with the quadratic form $x \mapsto \frac{1}{p}Q(x)$ (which may not be integral); then
\[A(\Omega, \Lambda; F|T^{(n)}(p)) = \chi([\Omega:p\Lambda])p^{E(\Omega, \Lambda)}a(\Omega^{1/p}; F).\]\end{theorem}

\begin{theorem}[Hafner--Walling, \cite{HafnerWalling2002}]\label{action-of-Tjp2-on-fourier-exp}  Let $F \in \mathcal{M}_k^{(n)}(N, \chi)$ have Fourier expansion (\ref{eqn:lattice-fourier-expansion}), let $p$ be any prime, let $0 \leq j \leq n$, and write 
\[(F|\widetilde{T}^{(n)}_j(p^2, \chi))(Z) = \sum_\Lambda a(\Lambda; F|\widetilde{T}^{(n)}_j(p^2, \chi))e\{\Lambda Z\}.\]  
Then
\[a(\Lambda; F|\widetilde{T}^{(n)}_j(p^2, \chi)) = \sum_{p\Lambda \subset \Omega \subset \frac{1}{p}\Lambda} A(\Omega, \Lambda; F|\widetilde{T}^{(n)}_j(p^2))\]
with $A(\Omega, \Lambda; F|\widetilde{T}^{(n)}_j(p^2)$ defined as follows: let $(m_0, m_1, m_2)$ be the $p$-type of $\Omega$ in $\Lambda$, and set
\[\begin{aligned} E_j(\Omega, \Lambda) &= k(m_0-m_2+j)+m_2(m_2+m_1+1)\\
&\qquad +\frac{(m_1-n+j)(m_1-n+j+1)}{2} - j(n+1);\end{aligned}\]
in the notation of Definition \ref{dfn:p-type} let $\alpha_j(\Omega, \Lambda)$ denote the number of totally isotropic subspaces of $\Lambda_1/p\Lambda_1$ of codimension $n-j$; then
\[A(\Omega, \Lambda; F|\widetilde{T}^{(n)}_j(p^2)) = \chi(p^{j-n}[\Omega:p\Lambda]) p^{E_j(\Omega, \Lambda)}\alpha_j(\Omega, \Lambda) a(\Omega; F).\]\end{theorem}

\noindent Let us finally record two simple results that we will frequently use.  The first is that the well-known fact that the reduction modulo $p$ map $\SL_n(\Z) \to \SL_n(\Z/p\Z)$ is surjective.  The second is the following simple corollary:

\begin{lemma}\label{trivial-lifting-lemma}  Let $\overline{G} = \left(\begin{smallmatrix} \overline{H} & 0 \\ \overline{B} & \overline{I}_m\end{smallmatrix}\right) \in \SL_n(\Z/p\Z)$, where $\overline{H} \in \SL_{n-m}(\Z/p\Z)$.  Let $H \in \SL_{n-m}(\Z)$ such that $H \bmod p = \overline{H}$.  Then we can take the lift $G\in \SL_n(\Z)$ of $\overline{G}$ to be of the form $\left(\begin{smallmatrix} H & 0 \\ B & I_m\end{smallmatrix}\right)$.  \end{lemma}
\begin{proof}  Let $B$ be any lift of $\overline{B}$, and consider $G = \left(\begin{smallmatrix} H & 0 \\ B & I_m\end{smallmatrix}\right)$.  Then $G \in \SL_n(\Z)$, and $G \bmod p = \overline{G}$.  \end{proof}

\section{The intertwining relations for $\Phi$ and $T^{(n)}_j(p^2)$}\label{sctn:first-relations}

For this section fix $n$ a positive integer and $1 \leq j \leq n$, and for ease of notation drop the character from the Hecke operator notation, so that $\widetilde{T}_j^{(n)}(p^2) = \widetilde{T}_j^{(n)}(p^2, \chi)$.  Let 
\[F(Z) = \sum_\Lambda a(\Lambda; F) e\{\Lambda Z\} \in \mathcal{M}_k^{(n)}(N, \chi).\]  
Applying the Hecke operator $\widetilde{T}_j^{(n)}(p^2)$ then the Siegel lowering operator $\Phi$ we obtain
\begin{equation}\label{Tjp2-then-phi}\Phi(F | \widetilde{T}_j^{(n)}(p^2))(Z') = \sum_{\Lambda'} \sum_{p\Lambda \subset \Omega \subset \frac{1}{p}\Lambda} A(\Omega, \Lambda; F|\widetilde{T}^{(n)}_j(p^2))e\{\Lambda' Z'\}\end{equation}
where $\Lambda$ varies over all rank $n$ lattices of the form $\Lambda' \oplus \Z x_n$, endowed with bilinear form $B$ obtained by extended the bilinear form $B'$ of $\Lambda'$ by the rule $B(x_n, y) = 0$ for all $y \in \Lambda$.  On the other hand, if we apply $\Phi$ first then $\widetilde{T}_j^{(n-1)}(p^2)$ (where we are now assuming $j \leq n-1$ as well) we obtain
\begin{equation}\label{phi-then-Tjp2}(\Phi(F) | \widetilde{T}_j^{(n-1)}(p^2))(Z') = \sum_{\Lambda'} \sum_{p\Lambda' \subset \Omega' \subset \frac{1}{p}\Lambda'} A(\Omega', \Lambda'; \Phi(F)|\widetilde{T}^{(n-1)}_j(p^2)) e\{\Lambda' Z'\}.\end{equation}
Proposition \ref{prop:tilde-intertwining-relations} in the sequel is an intertwining relation for the operators $\Phi$ and $\widetilde{T}_j^{(n)}(p^2)$.  We will prove this by comparing Fourier coefficients in (\ref{Tjp2-then-phi}) and (\ref{phi-then-Tjp2}).  We therefore fix a single lattice $\Lambda'$ of rank $n-1$ endowed with a bilinear form $B'$.  We write $\Lambda$ for the lattice $\Lambda' \oplus \Z x_n$ which is endowed with the bilinear form $B$ extending $B'$ as above.  A preliminary step in comparing the Fourier coefficients at $\Lambda'$ in (\ref{Tjp2-then-phi}) and (\ref{phi-then-Tjp2}) is to know which lattices $p\Lambda \subset \Omega \subset \frac{1}{p} \Lambda$ project on to a given $p\Lambda' \subset \Omega' \subset \frac{1}{p} \Lambda'$.  This is the content of Lemmas \ref{lattice-parameterisations} and \ref{lattice-counts}:

\begin{lemma}\label{lattice-parameterisations}  There is a one-to-one correspondence between: \begin{itemize}
\item  lattices $\Omega$ such that $p\Lambda \subset \Omega \subset \frac{1}{p}\Lambda$ with $p$-type $(t, s-t, n-s)$,
\item  the following data: \begin{itemize} \item  an $s$-dimensional subspace $\overline{\Delta_1}$ of $\Lambda/p\Lambda$.  Let $\Delta_1$ be the preimage of this in $\Lambda$,
\item  a $t$-dimensional subspace $\overline{\Delta_2}$ of $\Delta_1/p\Delta_1$, linearly independent of the subspace $\overline{p\Lambda}$ of $\Delta_1/p\Delta_1$. \end{itemize} \end{itemize} \end{lemma}
\begin{proof}  Suppose we are given $\Omega$ with $p\Lambda \subset \Omega \subset p^{-1}\Lambda$ and $p$-type $(t, s-t, n-s)$.  By the invariant factor theorem we can write 
\[\begin{aligned} \Lambda &= \Lambda_0 \oplus \Lambda_1 \oplus \Lambda_2, \\
\Omega &= \frac{1}{p}\Lambda_0 \oplus \Lambda_1 \oplus p\Lambda_2,\end{aligned}\] 
where $\rk(\Lambda_0) = t$, $\rk(\Lambda_1) = s-t$, $\rk(\Lambda_2)=n-s$.  Let $\Delta_1 = \Lambda \cap \Omega = \Lambda_0 \oplus \Lambda_1 \oplus p\Lambda_2$.  Then $\overline{\Delta_1} = \Delta_1 + p\Lambda \subset \Lambda/p\Lambda$ has dimension $s$.  Also, $p\Omega \subset \Delta_1$, and $\overline{\Delta_2} = p\Omega + p\Delta_1 \subset \Delta_1/p\Delta_1$ has dimension $t$, and is linearly independent of $\overline{p\Lambda} \subset \Delta_1/p\Delta_1$. \\

\noindent Conversely, suppose we pick a subspace $\overline{\Delta_1} \subset \Lambda/p\Lambda$ of dimension $s$; let $\Delta_1$ be its preimage in $\Lambda$.  Pick a basis $(\overline{y_1},...,\overline{y_s})$ for $\overline{\Delta_1}$ and extend to a basis $(\overline{y_1},...,\overline{y_n})$ of $\Lambda/p\Lambda$.  Note that $(\overline{x_1},...,\overline{x_n})$ is also a basis for $\Lambda/p\Lambda$, so there exists $\overline{G}_1 \in \GL_n(\Z/p\Z)$ such that $(\overline{y}_1,...,\overline{y_n}) = (\overline{x_1},...,\overline{x_n})\overline{G}_1$.  Replacing $\overline{y_1}$ by $\det(\overline{G}_1)^{-1}\overline{y_1}$ we may assume $\overline{G}_1 \in \SL_n(\Z/p\Z)$.  Since the projection map $\SL_n(\Z) \to \SL_n(\Z/p\Z)$ is surjective, we can pick $G_1 \in \SL_n(\Z)$ reducing modulo $p$ to $\overline{G}_1$.  Let $(y_1,...,y_n) = (x_1,...,x_n)G_1$, so $(y_1,...,y_n)$ is a basis for $\Lambda$ with $y_i$ reducing modulo $p$ to $\overline{y_i}$ and now
\begin{equation}\label{eqn:delta1-basis} \Delta_1 = \Z y_1 \oplus ... \oplus \Z y_s \oplus \Z py_{s+1} \oplus ... \oplus \Z py_n.\end{equation}
Note that, in $\Delta_1/p\Delta_1$, $\overline{p\Lambda} = p\Lambda + p\Delta_1$ has basis $(\overline{py_{s+1}},...,\overline{py_n})$.  Now pick a subspace $\overline{\Delta_2} \subset \Delta_1/p\Delta_1$ linearly independent of $\overline{p\Lambda}$.  Let $(\overline{z_1},...,\overline{z_t})$ be a basis for $\overline{\Delta_2}$.  Since $\{\overline{z_1},...,\overline{z_t}, \overline{py_{s+1}}, ..., \overline{py_n}\}$ is linearly independent, we can extend it to a basis $(\overline{z_1},...\overline{z_s}, \overline{py_{s+1}}, ..., \overline{py_n})$ for $\Delta_1/p\Delta$.  For future reference, call this extension step (*).  From (\ref{eqn:delta1-basis}) we have that $(\overline{y_1},...,\overline{y_s}, \overline{py_{s+1}}, ..., \overline{py_n})$ is a basis for $\Delta_1/p\Delta_1$.  So, modifiyng $\overline{z_1}$ if necessary as above, there is $\overline{G_2} \in \SL_n(\Z/p\Z)$ such that $(\overline{z_1},...\overline{z_s}, \overline{py_{s+1}}, ..., \overline{py_n}) = (\overline{y_1},...,\overline{y_s}, \overline{py_{s+1}}, ..., \overline{py_n})\overline{G_2}$.  In fact, we see $\overline{G_2} = \left(\begin{smallmatrix} \overline{H} & 0 \\ \overline{B} & I_{n-s} \end{smallmatrix}\right)$ for some $\overline{H} \in \SL_s(\Z/p\Z)$.  Pick a lift $H \in \SL_s(\Z)$ of $\overline{H}$.  Using Lemma \ref{trivial-lifting-lemma}, choose a lift $G_2 \in \SL_n(\Z)$ of $\overline{G_2}$ of the form $\left(\begin{smallmatrix} H & 0 \\ B & I \end{smallmatrix}\right)$.  Let $(z_1,...,z_s, py_{s+1},...,py_n) = (y_1,...,y_s, py_{s+1},...,py_n)G_2$; thus $(z_1,...,z_s, py_{s+1},...,py_n)$ is a basis for $\Delta_1$, the $z_i$ reduce modulo $p\Delta_1$ to $\overline{z_i}$, and the preimage of $\overline{\Delta_2}$ in $\Delta_1$ is
\[\Z z_1 \oplus ... \oplus \Z z_t \oplus \Z pz_{t+1} \oplus ... \oplus \Z py_s \oplus \Z p^2y_{s+1} \oplus ... \oplus \Z p^2y_n.\]
Recall $G_2 = \left(\begin{smallmatrix} H & 0 \\ B & I \end{smallmatrix}\right)$.  Then $G_2' = \left(\begin{smallmatrix} H & 0 \\ pB & I \end{smallmatrix}\right) \in \SL_n(\Z)$ as well, and we have 
\[(z_1,...,z_s, y_{s+1},...,y_n) = (x_1,...,x_n)G_1G_2'.\]  
Thus $(z_1,...,z_s, y_{s+1},...,y_n)$ is a basis for $\Lambda$, and we can consider the lattice
\[\Omega = \Z\left(\frac{1}{p}z_1\right) \oplus ... \oplus \Z\left(\frac{1}{p}z_t\right) \oplus \Z z_{t+1} \oplus ... \oplus \Z z_s \oplus \Z py_{s+1} \oplus ... \oplus \Z py_n.\]
Note that this construction is independent of the choice of $(\overline{y_1},...,\overline{y_n})$ and $(\overline{z_1},...,\overline{z_t})$.\\

\noindent We have therefore constructed maps between the two pieces of data, and they are easily seen to be inverse to each other.  \end{proof}

\begin{corollary}\label{counting-corollary}  The number of lattices $\Omega$ with $p\Lambda \subset \Omega \subset p^{-1}\Lambda$ and $p$-type $(t, s-t, n-s)$ is $\left(\frac{n}{s}\right)_p\left(\frac{s}{t}\right)_p p^{t(n-s)}$.  \end{corollary}
\begin{proof}  $\left(\frac{n}{s}\right)_p$ counts the number of $s$-dimensional subspaces of $\Lambda/p\Lambda$, and $\left(\frac{s}{t}\right)_p p^{t(n-s)}$ counts the number of $t$-dimensional subspaces of $\Delta_1/p\Delta_1$ linearly independent of $\overline{p\Lambda} \subset \Delta_1/p\Delta_1$.  \end{proof}

\begin{lemma}\label{lattice-counts}   Let $\Omega'$ be a lattice with $p\Lambda' \subset \Omega' \subset p^{-1}\Lambda'$ and $p$-type $(l, r-l, n-r-1)$.  Recall that $\Lambda = \Lambda' \oplus \Z x_n$.  Then under the projection $\Lambda \to \Lambda'$ the lattices $\Omega$ with $p\Lambda \subset \Omega \subset \frac{1}{p}\Lambda$ that project on to $\Omega'$ are classified as follows: \renewcommand{\labelenumi}{(\Alph{enumi})}\begin{enumerate}
\item  one lattice with $p$-type $(l+1, r-l, n-r-1)$, which (following the proof) we will denote $\Omega^{(1)}$.
\item  $p^l$ lattices with $p$-type $(l, r-l+1, n-r-1)$, which we will denote $\Omega^{(2)}((\overline{\alpha_i})_{1 \leq i \leq l})$ where $\overline{\alpha_i} \in \Z/p\Z$ for $1 \leq i \leq l$.
\item  $p^{l+r}$ lattices with $p$-type $(l, r-l, n-r)$, which we will denote $\Omega^{(3)}((\overline{\alpha_i})_{1 \leq i \leq r})$ where $\overline{\alpha_i} \in \Z/p^2\Z$ for $1 \leq i \leq l$ and $\overline{\alpha_i} \in \Z/p\Z$ for $l+1 \leq i \leq r$. 
\item  for each of the $p^{r-l}-1$ non-zero vectors $\overline{u'} \in \Lambda'_1/p\Lambda'_1$, $p^l$ lattices with $p$-type $(l+1, r-l-1, n-r)$, which we will denote $\Omega^{(4)}(\overline{u'}, (\overline{\gamma_i})_{1 \leq i \leq l})$.  \end{enumerate}
Moreover, let $\Omega$ be such a lattice projecting on to $\Omega'$.  Write
\begin{equation}  \Lambda' = \Lambda_0' \oplus \Lambda_1' \oplus \Lambda_2',\:\text{and } \Omega' = \frac{1}{p}\Lambda_0' \oplus \Lambda_1' \oplus p \Lambda_2',\end{equation} 
\begin{equation}\label{eqn:upstairs-lattice-IFT}  \Lambda = \Lambda_0 \oplus \Lambda_1 \oplus \Lambda_2\,\:\text{and } \Omega = \frac{1}{p}\Lambda_0 \oplus \Lambda_1 \oplus p\Lambda_2. \end{equation}
Then we have the following characterisation of $\Lambda_1/p\Lambda_1$ in each case:\renewcommand{\labelenumi}{(\Alph{enumi})}\begin{enumerate}
\item  For $\Omega = \Omega^{(1)}$, $\Lambda_1/p\Lambda_1 = \Lambda'_1/p\Lambda'_1$.
\item  For any $\Omega = \Omega^{(2)}((\overline{\alpha_i})_{1 \leq i \leq l})$, $\Lambda_1/p\Lambda_1 = \Lambda'_1/p\Lambda'_1 \oplus (\Z/p\Z) \overline{x_n}$.
\item  For any $\Omega = \Omega^{(3)}((\overline{\alpha_i})_{1 \leq i \leq r})$, $\Lambda_1/p\Lambda_1 = \Lambda'_1/p\Lambda'_1$.
\item  For $\Omega = \Omega^{(4)}(\overline{u'}, (\overline{\gamma_i})_{1 \leq i \leq l})$, $\Lambda_1/p\Lambda$ is a codimension one subspace of $\Lambda'/p\Lambda'$ which does not contain $\overline{u'}$.
\end{enumerate}\end{lemma}

\begin{proof}  We follow the construction of Lemma \ref{lattice-parameterisations}.  First pick the subspace $\overline{\Delta_1}$, there are two possibilities:
\begin{enumerate}  \item  $\overline{x_n} \in \overline{\Delta_1}$.  We may assume $\overline{y_s} = \overline{x_n}$, and choosing our the lifting matrix $G_1$ with the aid of Lemma \ref{trivial-lifting-lemma}, we may also assume that $y_s = x_n$, so that
\[\Delta_1 = \Z y_1 \oplus ... \oplus \Z y_{s-1} \oplus \Z x_n \oplus \Z py_{s+1} \oplus ... \oplus \Z py_n.\]
Here each $y_i \in \Lambda$.  Recall that $\Lambda = \Lambda' \oplus \Z{x_n}$.  For $s+1 \leq i \leq n$ write $y_i = y_i' + \alpha_i x_n$ where $y_i' \in \Lambda'$.  Since $x_n \in \Delta_1$ we may assume $\alpha_i=0$ for $s+1 \leq i \leq n$ (we could also do this for $1 \leq i \leq s$, but it is convenient not to for now).  Thus $y_i = y_i' \in \Lambda'$ and we have
\[\Delta_1 = \Z y_1 \oplus ... \oplus \Z y_{s-1} \oplus \Z x_n \oplus \Z py_{s+1}' \oplus ... \oplus \Z py_n'.\]
We now pick $\overline{\Delta_2}$: \begin{enumerate}
\item  $\overline{x_n} \in \overline{\Delta_2}$.  We may assume $\overline{z_t} = \overline{x_n}$, and choosing our lifting matrix $G_2$ (or, more precisely, $H$) appropriately we may also assume that $z_t = x_n$.  This constructs the lattice
\[\begin{aligned} \Omega^{(1)} &= \Z\left(\frac{1}{p}z_1\right) \oplus ... \oplus \Z\left(\frac{1}{p}z_{t-1}\right) \oplus \Z\left(\frac{1}{p}x_n\right) \\
&\qquad \oplus \Z z_{t+1} \oplus ... \oplus \Z z_s \oplus \Z py_{s+1}' \oplus ... \oplus \Z py_n'.\end{aligned}\]
Since the $z_i$ are in $\Lambda$ we can write $z_i = z_i' + \alpha_i x_n$ where $z_i' \in \Lambda'$.  Since $(1/p)x_n \in \Omega^{(1)}$ we may assume all $\alpha_i=0$.  Thus our lattice is
\[\begin{aligned} \Omega^{(1)} &= \Z\left(\frac{1}{p}z_1'\right) \oplus ... \oplus \Z\left(\frac{1}{p}z_{t-1}'\right) \oplus \Z\left(\frac{1}{p}x_n\right) \\
&\qquad \oplus \Z z_{t+1}' \oplus ... \oplus \Z z_s' \oplus \Z py_{s+1}' \oplus ... \oplus \Z py_n' \end{aligned}\]
and this projects to
\[\begin{aligned}\Omega^{(1)}{}' &= \Z\left(\frac{1}{p}z_1'\right) \oplus ... \oplus \Z\left(\frac{1}{p}z_{t-1}'\right) \\
&\qquad \oplus \Z z_{t+1}' \oplus ... \oplus \Z z_s' \oplus \Z py_{s+1}' \oplus ... \oplus \Z py_n'.\end{aligned}\]
\item  $\overline{x_n} \notin \overline{\Delta_2}$.  Let $\overline{z_1},...,\overline{z_t}$ be a basis for $\overline{\Delta_2}$ and recall $\overline{py_{s+1}'},...,\overline{py_n'}$ is a basis for $p\Lambda \subset \Delta_1/p\Delta$ as in Lemma \ref{lattice-parameterisations}; and moreover that $\{\overline{z_1},...,\overline{z_t}, \overline{py_{s+1}'},...\overline{py_n'}\}$ is linearly independent.  There are two possibilities:
\begin{enumerate}
\item  $\{\overline{z_1},....\overline{z_t}, \overline{py_{s+1}'},...,\overline{py_n'}, \overline{x_n}\}$ is linearly independent.  So when we extend to a basis $(\overline{z_1},...,\overline{z_s}, \overline{py_{s+1}'},...,\overline{py_n}')$ for $\Delta_1/p\Delta_1$ at step (*) in the proof of Lemma \ref{lattice-counts}, we can include $\overline{x_n}$ in this extension, say $\overline{z_s} = \overline{x_n}$.  Choosing the lifting matrix $G_2$ appropriately we may assume $z_s = x_n$ as well.  Then we have the lattice
\[\begin{aligned} \Omega^{(2)}&= \Z\left(\frac{1}{p}z_1\right) \oplus ... \oplus \Z\left(\frac{1}{p}z_t\right) \oplus \Z z_{t+1} \oplus ... \oplus \Z z_{s-1} \\
&\qquad \oplus \Z x_n \oplus \Z py_{s+1}' \oplus ... \oplus \Z py_{n}'.\end{aligned}\]
Again write $z_i = z_i' + \alpha_i x_n$ where $z_i' \in \Lambda'$.  Since $x_n \in \Lambda'$ we may assume $\alpha_i = 0$ for $t+1 \leq i \leq s-1$, and $\alpha_i \in \{0,...,p-1\}$ for $1 \leq i \leq t$.  Hence our lattice is
\[\begin{aligned} \Omega^{(2)}((\alpha_i)_{1 \leq i \leq t}) &= \Z\left(\frac{1}{p}(z_1' + \alpha_1 x_n)\right) \oplus ... \oplus \Z\left(\frac{1}{p}(z_t' + \alpha_t x_n)\right) \\
&\qquad\oplus \Z z_{t+1}' \oplus ... \oplus \Z z_{s-1}' \oplus \Z x_n \\
&\qquad\oplus \Z py_{s+1}' \oplus ... \oplus \Z py_{n}'\end{aligned}\]
and, for any choice of $(\alpha_i)$, this projects to
\[\begin{aligned} \Omega^{(2)}{}' &= \Z\left(\frac{1}{p}z_1'\right) \oplus ... \oplus \Z\left(\frac{1}{p}z_t'\right) \oplus \Z z_{t+1}' \oplus ... \oplus \Z z_{s-1}' \\
&\qquad \oplus \Z py_{s+1}' \oplus ... \oplus \Z py_{n}'.\end{aligned}\]
\item  $\{\overline{z_1},....\overline{z_t}, \overline{py_{s+1}'},...,\overline{py_n'}, \overline{x_n}\}$ is linearly dependent, so we have a relation $\overline{x_n} = \sum a_i \overline{z_i} + \sum b_i \overline{py_i'}$.  If all the $a_i$ are $0$ then $\overline{x_n} \in \overline{p\Lambda}$ which is a contradiction; and if all the $b_i$ are $0$ then $\overline{x_n} \in \overline{\Delta_2}$ which is also a contradiction.  Modifying the basis $\{\overline{z_1},...,\overline{z_t}\}$ for $\overline{\Delta_2}$, we may therefore assume $\overline{x_n} = \overline{z_t} - \overline{pu'}$ for some non-zero $\overline{pu'} \in \bigoplus_{i=s+1}^n \mathbb{F}_p \overline{py_i}$, or $\overline{z_t} = \overline{x_n} + \overline{pu'}$.  Extend to a basis $(\overline{z_1},...,\overline{z_s}, \overline{py_{s+1}'},...,\overline{py_n'})$ for $\Delta_1/p\Delta$ as in step (*) of in the proof of Lemma \ref{lattice-parameterisations}.  Pick some lift $u'$ of $\overline{u'}$.  Recall that $(x_1,...,x_n)$ is our basis for $\Lambda$, and that $u' \in \Lambda'$ where $\Lambda = \Lambda' \oplus \Z x_n$, so $(x_1,...,x_{n-1}, x_n + pu')$ is also a basis for $\Lambda$.  Note that $(x_n + pu') + p\Delta_1 = \overline{z_t}$.  We can then choose a lifting matrix appropriately with respect to this basis to ensure that $z_t = x_n + pu'$ is a basis vector of $\Omega$, so that our lattice is
\[\begin{aligned} \Omega^{(4)}(u') &= \Z\left(\frac{1}{p}z_1\right) \oplus ... \oplus \Z\left(\frac{1}{p}z_{t-1}\right) \oplus \Z\left(\frac{1}{p}x_n + u'\right) \oplus \\
&\qquad \Z z_{t+1} \oplus ... \oplus \Z z_s \oplus \Z py_{s+1}' \oplus ... \oplus \Z py_n'.\end{aligned}\]
Write each $z_i = z_i' + \alpha_i x_n$ where $z_i' \in \Lambda'$.  Note that for $t+1 \leq i \leq s$ we have 
\[z_i' = (z_i' + \alpha_ix_n) - \alpha_i p\left(\frac{1}{p}x_n + y_n'\right) + \alpha_i py_n'\]
so we can assume $\alpha_i = 0$ for $t+1 \leq i \leq s$.  Similarly we may assume $\alpha_i \in \{0,...,p-1\}$ for $1 \leq i \leq t$.  Then our lattice is
\[\begin{aligned} &\Omega^{(4)}(u', (\alpha_i)_{1 \leq i \leq t-1}) \\
&\qquad= \Z\left(\frac{1}{p}(z_1' + \alpha_1 x_n)\right) \oplus ... \oplus \Z\left(\frac{1}{p}(z_{t-1}' + \alpha_{t-1} x_n)\right) \\
&\qquad\qquad \oplus \Z\left(\frac{1}{p}x_n + u'\right) \oplus \Z z_{t+1}' \oplus ... \oplus \Z z_s' \\
&\qquad\qquad \oplus \Z py_{s+1}' \oplus ... \oplus \Z py_n'\end{aligned}\] 
and, for any choice of $(\alpha_i)$, this projects to
\[\begin{aligned} \Omega^{(4)}(u')' &= \Z\left(\frac{1}{p}z_1'\right) + ... + \Z\left(\frac{1}{p}z_{t-1}'\right) + \Z u' \\
&\qquad + \Z z_{t+1}' + ... + \Z z_s' + \Z py_{s+1}' + ... + \Z py_{n}'. \end{aligned}\]\end{enumerate} \end{enumerate}

\item In contrast to 1. we now have $\overline{x_n} \notin \overline{\Delta_1}$.  Pick a basis $\{\overline{y_1},...,\overline{y_s}\}$ for $\overline{\Delta_1}$.  When we extend to a basis for $\Lambda/p\Lambda$ we may assume $\overline{x_n}$ is included in that extension, say $\overline{y_n} = \overline{x_n}$.  Choosing our lifting matrix $G_2$ with the aid of Lemma \ref{trivial-lifting-lemma} we may assume $y_n = x_n$.  Follow through the rest of the construction as in Lemma \ref{lattice-counts}, we construct the lattice
\[\begin{aligned}\Omega^{(3)} &= \Z\left(\frac{1}{p}z_1\right) \oplus ... \oplus \Z\left(\frac{1}{p}z_t\right) \oplus \Z z_{t+1} \oplus ... \oplus \Z z_s \\
&\qquad\oplus \Z py_{s+1} \oplus ... \oplus \Z py_{n-1} \oplus \Z px_n.\end{aligned}\]
Write each $z_i = z_i' + \alpha_i x_n$, $y_i = y_i' = \alpha_i x_n$ where $x_i', y_i' \in \Lambda'$.  Since $px_n \in \Omega^{(3)}$, we may assume $\alpha_i = 0$ for $i \geq s+1$, $\alpha_i = \{0,...,p-1\}$ for $t+1 \leq i \leq s$ and $\alpha_i \in \{0,...,p^2-1\}$ for $1 \leq i \leq t$.  Then we have
\[\begin{aligned} \Omega^{(3)}((\alpha_i)_{1 \leq i \leq s}) &= \Z\left(\frac{1}{p}(z_1' + \alpha_1 x_n)\right) \oplus ... \oplus \Z\left(\frac{1}{p}(z_t' + \alpha_t x_n)\right) \\
&\qquad \oplus \Z(z_{t+1}' + \alpha_{t+1} x_n) \oplus ... \oplus \Z(z_s' + \alpha_s x_n) \\
&\qquad \oplus \Z py_{s+1}' \oplus ... \oplus \Z py_{n-1}' \oplus \Z px_n\end{aligned}\]
and, for any choice of $(\alpha_i)$, this projects on to
\[\Omega^{(3)}{}' = \Z\left(\frac{1}{p}z_1'\right) \oplus ... \oplus \Z\left(\frac{1}{p}z_t'\right) \oplus \Z z_{t+1}' \oplus ... \oplus \Z z_s' \oplus \Z py_{s+1}' \oplus ... \oplus \Z py_{n-1}'.\] \end{enumerate} 

\noindent Now fix a lattice $\Omega'$ with $p$-type $(l, r-l, n-r-1)$.  We consider in the following cases how many lattices project on to $\Omega'$, what their $p$-types are, and the structure of their $\Lambda_1/p\Lambda_1$ part in (\ref{eqn:upstairs-lattice-IFT}):
\renewcommand{\labelenumi}{(\Alph{enumi})}\begin{enumerate}
\item  Consider case 1(a).  Here we see that, since $\Omega^{(1)}{}'$ has $p$-type $(l, r-l, n-r-1)$, $\Omega^{(1)}$ must have $p$-type $(l+1, r-l, n-r-1)$.  Also, $\Omega^{(1)}$ is uniquely determined by $\Omega^{(1)}{}'$.  Finally, by inspection we see that $\Lambda_1/p\Lambda_1 = \Lambda_1'/p\Lambda_1'$.  
\item  Consider case 1(b)(i).  Here we see that $\Omega^{(2)}((\alpha_i))$ must have $p$-type $(l, r-l+1, n-r-1)$, and there are $p^l$ lattices with the same projection $\Omega^{(2)}{}'$.  Moreover, $\Lambda_1/p\Lambda_1 = \Lambda'_1/p\Lambda'_1 \oplus (\Z/p\Z)\overline{x_n}$.
\item  Consider case 2.  Here we see that $\Omega^{(3)}((\alpha_i))$ must have $p$-type $(l, r-l, n-r)$, and there are $p^{r+l}$ lattices with the same projection $\Omega^{(3)}{}'$.  Moreover, $\Lambda_1/p\Lambda_1 = \Lambda'_1/p\Lambda'_1$.
\item  Consider case 1(b)(ii).  Since $\Omega^{(4)}(u')'$ has $p$-type $(l, r-l, n-r-1)$ we see that $\Omega^{(4)}(u', (\alpha_i))$ must have $p$-type $(l+1, r-l-1, n-r)$.  Also there are $p^l$ lattices with the same projection $\Omega^{(4)}(u')'$, and by inspection we see that for these lattices $\Lambda_1/p\Lambda_1$ is a codimension $1$ subspace of $\Lambda'_1/p\Lambda'_1$ which does not contain $\overline{u'}$.

We now describe some cases when different choices of the vector $u'$ give different lattices with the same projection.  Following this, we will prove that, after taking this in to account, we have constructed all lattices projecting on to $\Omega'$.  First note that $\Omega^{(4)}(u_1', (\alpha_i)) = \Omega^{(4)}(u_2', (\beta_i))$ if and only if $(\alpha_i) = (\beta_i)$ and $u_1' - u_2' \in p\Lambda'$.  Now fix a basis for the projection
\[\Omega' = \Z\left(\frac{1}{p}w_1'\right) \oplus ... \oplus \Z\left(\frac{1}{p}w_l'\right) \oplus \Z w_{l+1}' \oplus ... \oplus \Z w_r' \oplus \Z pw_{r+1}' \oplus ... \oplus \Z pw_{n-1}'.\]
Take $u' = a_1w_{l+1}' + ... + a_r w_r'$ to be any vector such that $u' \notin p\Lambda'$.  We easily see that, for any choice of $(\alpha_i)$, $\Omega^{(4)}(u', (\alpha_i))' = \Omega'$.  As $u'$ varies such that $u' + p\Lambda'$ covers all $p^{r-l}-1$ non-zero possibilities, we obtain $p^l(p^{r-l}-1)$ distinct lattices $\Omega^{(4)}(u', (\alpha_i))$ all projecting on to $\Omega'$.  
\end{enumerate}

We have now listed all possible rank $n$ lattices projecting on to $\Omega'$.  Note that these lattice are all distinct: indeed, the lattices within each case are distinct by construction, and there can be no equality between two lattices in different cases since the $p$-type of their projections are different.  This completes the proof.  \end{proof}

\begin{remark}  Let us demonstrate the consistency of the numbers from Lemma \ref{lattice-counts} by counting the number $M(t, s-t, n-s)$ of rank $n$ lattices with $p$-type $(t, s-t, n-s)$: on the one hand this is equal to $\left(\frac{n}{s}\right)_p \left(\frac{s}{t}\right)_p p^{t(n-s)}$, by Corollary \ref{counting-corollary}.  On the other hand, using Lemma \ref{lattice-counts}, it is equal to 

\[\begin{aligned} &M(t, s-t, n-s) \\
&\qquad = \left(\frac{n-1}{s-1}\right)_p\left(\frac{s-1}{t-1}\right)_p p^{(t-1)(n-s)} + p^t\left(\frac{n-1}{s-1}\right)_p \left(\frac{s-1}{t}\right)_p p^{t(n-s)} \\
&\qquad\qquad + p^{t+s}\left(\frac{n-1}{s}\right)_p \left(\frac{s}{t}\right)_p p^{t(n-s-1)} \\
&\qquad\qquad + (p^{s}-p^{t-1})\left(\frac{n-1}{s}\right)_p \left(\frac{s}{t-1}\right)_p p^{(t-1)(n-s-1)} \\
&\qquad= p^{t(n-s)}\left[ \left(\frac{n-1}{s-1}\right)_p \left(\frac{s-1}{t-1}\right)_p p^{-n+s} + \left(\frac{n-1}{s-1}\right)_p \left(\frac{s-1}{t}\right)_p p^t \right.\\
&\left.\qquad\qquad + \left(\frac{n-1}{s}\right)_p \left(\frac{s}{t}\right)_p p^s + \left(\frac{n-1}{s}\right)_p \left(\frac{s}{t-1}\right)_p p^{-n+s}(p^s - p^{t-1}) \right].\end{aligned}\]
It is then straightforward using the properties of the Gaussian binomial coefficient to prove that the right hand side is equal to $p^{t(n-s)} \left(\frac{n}{s}\right)_p \left(\frac{s}{t}\right)_p$.
\end{remark}

\begin{proposition}\label{prop:tilde-intertwining-relations}  Let $F \in \mathcal{M}_k(N, \chi)$, $1 \leq j \leq n$, and let $\Lambda$ be a $\Z$-lattice with a $\Z$-valued quadratic form.  Then
\[\begin{aligned} \Phi(F | \widetilde{T}^{(n)}_j(p^2)) &= \Phi(F)|\widetilde{T}_j^{(n-1)}(p^2) + \widetilde{c}_{j, j-1}^{(n-1)} \Phi(F) | \widetilde{T}_{j-1}^{(n-1)}(p^2) \\
&\qquad + \widetilde{c}_{j, j-2}^{(n-1)}\Phi(F) | \widetilde{T}_{j-2}^{(n-1)}(p^2)\end{aligned}\]
where
\[\begin{aligned} \widetilde{c}_{j, j-1}^{(n-1)} &= \chi(p^2)p^{2k-j-n} + \chi(p)p^{k-n} + p^{n-j}, \\
\widetilde{c}_{j, j-2}^{(n-1)} &= \chi(p^2)(p^{2k-2j+1} - p^{2k-n-j}).\end{aligned}\]
We adopt the convention that $\widetilde{T}_j^{(n-1)}(p^2)$ is the zero operator for $j \in \{n, -1, -2\}$.
\end{proposition}

\begin{proof}Continue with the fixed lattice $\Lambda'$, and the lattice $\Lambda = \Lambda' \oplus \Z x_n$ with the quadratic form extended as above.  It suffices to show that
\[\begin{aligned} \sum_{p\Lambda \subset \Omega \subset \frac{1}{p} \Lambda} A(\Omega, \Lambda; F|\widetilde{T}_j^{(n)}(p^2)) &= \sum_{p\Lambda' \subset \Omega' \subset \frac{1}{p}\Lambda'} A(\Omega', \Lambda'; \Phi(F)|\widetilde{T}_j^{(n-1)}(p^2)) \\
&\qquad+ \widetilde{c}_{j, j-1}^{(n-1)}\sum_{p\Lambda' \subset \Omega' \subset \frac{1}{p}\Lambda'} A(\Omega', \Lambda'; \Phi(F)|\widetilde{T}^{(n-1)}_{j-1}(p^2)) \\
&\qquad + \widetilde{c}_{j, j-2}^{(n-1)} \sum_{p\Lambda' \subset \Omega' \subset \frac{1}{p}\Lambda'} A(\Omega', \Lambda'; \Phi(F)|\widetilde{T}_{j-2}^{(n-1)}(p^2))\end{aligned}\]
Write $\pi$ for map of Lemma \ref{lattice-counts} (i.e. the projection $x_n \mapsto 0$).  For $\Omega'$ a rank $n-1$ lattice set
\[B(\Omega', \Lambda'; F|\widetilde{T}^{(n)}_j(p^2)) = \sum_{\Omega\text{ s.t.}\pi(\Omega) = \Omega'} A(\Omega, \Lambda; F|\widetilde{T}^{(n)}_j(p^2)).\]
So
\[\sum_{p\Lambda \subset \Omega \subset \frac{1}{p}\Lambda} A(\Omega, \Lambda; F|\widetilde{T}_j^{(n)}(p^2)) = \sum_{p\Lambda' \subset \Omega' \subset \frac{1}{p}\Lambda'} B(\Omega', \Lambda';F|\widetilde{T}_j^{(n)}(p^2)),\] and it suffices to show that
\begin{equation}\label{eqn:suffices-to-prove} \begin{aligned} B(\Omega', \Lambda'; F|\widetilde{T}_j^{(n)}(p^2)) &= A(\Omega', \Lambda'; \Phi(F)|\widetilde{T}_j^{(n-1)}(p^2)) \\
&\qquad + \widetilde{c}_{j, j-1}^{(n-1)}A(\Omega', \Lambda'; \Phi(F)|\widetilde{T}_{j-1}^{(n-1)}(p^2))\\ 
&\qquad + \widetilde{c}_{j, j-2}^{(n-1)}A(\Omega', \Lambda'; \Phi(F)|\widetilde{T}_{j-2}^{(n-1)}(p^2))\end{aligned}\end{equation}
for each $p\Lambda' \subset \Omega' \subset \frac{1}{p}\Lambda'$.\\

Take such an $\Omega'$, say with $p$-type $(l, r-l, n-r-1)$.  Then the $\Omega$ such that $\pi(\Omega) = \Omega'$ are described by Lemma \ref{lattice-counts}.  Working from the notation of Lemma \ref{lattice-counts}, let us write $\Omega^{(2)}$ for any lattice of the form $\Omega^{(2)}((\overline{\alpha_i}))$, $\Omega^{(3)}$ any lattice of the form $\Omega^{(3)}((\overline{\alpha_i}))$, and $\Omega^{(4)}(\overline{u'})$ any lattice of the form $\Omega^{(4)}(\overline{u'}, (\overline{\alpha_i}))$.  Then it is easy to see that
\begin{equation}\label{eqn:iso-omega1} \alpha^{(n)}_j(\Omega^{(1)}, \Lambda) = \alpha^{(n-1)}_{j-1}(\Omega', \Lambda')\end{equation}
and
\begin{equation}\label{eqn:iso-omega3} \alpha^{(n)}_j(\Omega^{(3)}, \Lambda) = \alpha^{(n-1)}_{j-1}(\Omega', \Lambda').\end{equation}
Indeed, by Lemma \ref{lattice-counts} we have, for $\Omega = \Omega^{(1)}$, $\Lambda_1/p\Lambda_1 = \Lambda_1'/p\Lambda_1'$.  Thus $\alpha^{(n)}_j(\Omega^{(1)}, \Lambda)$ counts the number of codimesnion $n-j$ totally isotropic subspaces of $\Lambda_1'/p\Lambda_1'$.  But $\alpha^{(n-1)}_{j-1}(\Omega', \Lambda')$ also counts the number of codimension $(n-1)-(j-1) = n-j$ totally istropoic subspace of $\Lambda_1'/p\Lambda_1'$.  The same argument works for $\Omega^{(3)}$.\\

\noindent For $\Omega = \Omega^{(2)}$ we have $\Lambda_1/p\Lambda_1 = \Lambda'_1/p\Lambda'_1 \oplus (\Z/p\Z)\overline{x_n}$, and $\alpha^{(n)}_j(\Omega^{(2)}, \Lambda)$ counts the number of codimension $n-j$ totally isotropic subspaces of this space.  $\Omega'$ has $p$-type $(l, r-l, n-r-1)$ so $\Lambda_1/p\Lambda_1$ has dimension $r-l+1$, so a codimension $n-j$ subspace is a dimension $r-l+1-n+j$ subspace.  Recall that the line $(\Z/p\Z)\overline{x_n}$ is isotropic.  A totally isotropic subspace of $\Lambda_1/p\Lambda_1$ of dimension $r-l-n+j+1$ is therefore either the direct sum $(\Z/p\Z)\overline{x_n}$ with a dimension $r-l-n+j$ subspace of $\Lambda'/p\Lambda'$ (of which there are $\alpha^{(n-1)}_{j-1}(\Lambda', \Omega')$); or is formed by picking a totally isotropic subspace of $\Lambda'/p\Lambda'$ of dimension $r-l-n+j+1$ (of which there are $\alpha^{(n-1)}_j(\Lambda', \Omega')$) and adding some $\overline{\alpha}\overline{x_n}$ ($\overline{\alpha} \in (\Z/p\Z)$) to each basis vector.   We therefore have
\begin{equation}\label{eqn:iso-omega2}  \alpha^{(n)}_j(\Omega^{(2)}, \Lambda) = p^{r-l-n+j+1} \alpha^{(n-1)}_j(\Omega', \Lambda') + \alpha^{(n-1)}_{j-1}(\Omega', \Lambda').\end{equation}
Finally, consider $\sum_{\overline{u'}} \alpha^{(n)}_j(\Omega^{(4)}(\overline{u'}), \Lambda)$.  For $\Omega = \Omega^{(4)}(\overline{u'})$, $\Lambda/p\Lambda$ is a codimension $1$ subspace of $\Lambda'/p\Lambda'$ which does not contain $\overline{u'}$.  $\alpha^{(n)}_j(\Omega^{(4)}(\overline{u'}), \Lambda)$, which counts the number of totally isotropic codimension $n-j$ subspaces of $\Lambda/p\Lambda$, therefore counts totally isotropic subspaces of $\Lambda/p\Lambda$ of dimension $r-l-n+j-1$.  Subspaces of this dimension in $\Lambda'/p\Lambda'$ are counted by $\alpha^{(n-1)}_{j-2}(\Omega', \Lambda')$.  Let $\overline{V}$ be a totally istropic subspace of $\Lambda'/p\Lambda'$ of dimension $r-l-n+j-1$; we will consider how many times $\overline{V}$ is counted in $\sum_{\overline{u'}} \alpha^{(n)}_j(\Omega^{(4)}(\overline{u'}), \Lambda)$.  For a fixed choice of nonzero $\overline{u'} \in \Lambda'/p\Lambda'$ we see that $\overline{V}$ is counted by $\alpha^{(n)}_j(\Omega^{(4)}(\overline{u'}), \Lambda)$ if and only if $\overline{u'} \notin \overline{V}$.  So the number of times $\overline{V}$ is counted in $\sum_{\overline{u'}} \alpha^{(n)}_j(\Omega^{(4)}(\overline{u'}), \Lambda)$ is precisely the number of nonzero vectors $\overline{u'} \in \Lambda'/p\Lambda'$ that are not contained in $\overline{V}$.  Since $\overline{V}$ has codimension $n-j+1$, the number of such $\overline{u'}$ is $p^{n-j+1}-1$.  We therefore have
\begin{equation}\label{eqn:iso-omega4} \sum_{\overline{u'}} \alpha^{(n)}_j(\Omega^{(4)}(\overline{u'}), \Lambda) = (p^{n-j+1}-1)\alpha^{(n-1)}_{j-2}(\Omega', \Lambda'). \end{equation}

Now the remaining quantities appearing in $A(\Omega, \Lambda; F|\widetilde{T}_j^{(n)}(p^2)$ depend only on the $p$-type of $\Omega$ in $\Lambda$.  Using this observation and the above computations together with the count of Lemma \ref{lattice-counts} we can write
\begin{equation}\label{eqn:B-using-lattice-count} \begin{aligned} B(\Omega', \Lambda'; F|\widetilde{T}^{(n)}_j(p^2)) &= A(\Omega^{(1)}, \Lambda; F|\widetilde{T}^{(n)}_j(p^2)) \\
&\qquad+ p^l A(\Omega^{(2)}, \Lambda; F|\widetilde{T}^{(n)}_j(p^2)) \\
&\qquad+ p^{r+l} A(\Omega^{(3)}, \Lambda; F|\widetilde{T}^{(n)}_j(p^2)) \\
&\qquad+ p^l \sum_{\overline{u'}} A(\Omega^{(4)}(\overline{u'}), \Lambda; F|\widetilde{T}^{(n)}_j(p^2)). \end{aligned} \end{equation}
Now the appearance to the subscript $j$ on the right hand side of (\ref{eqn:iso-omega2}) suggests that we should consider $\Omega^{(2)}$ first: one easily computes from 
\[\begin{aligned} E_j(\Omega^{(2)}, \Lambda) &= n-r-j-1 + E_j(\Omega', \Lambda') \\
\chi(p^{j-n}[\Omega^{(2)}:p\Lambda]) &= \chi(p^{j-n+1}[\Omega':p\Lambda']),\end{aligned}\] 
and (\ref{eqn:iso-omega2}) that
\[\begin{aligned} p^l A(\Omega^{(2)}, \Lambda; \widetilde{T}^{(n)}_j(p^2)) &= A(\Omega', \Lambda'; \Phi(F) | \widetilde{T}^{(n-1)}_j(p^2)) \\
&\qquad + p^l \chi(p^{j-n}[\Omega^{(2)}:p\Lambda]) p^{E_j(\Omega^{(2)}, \Lambda)} \alpha_{j-1}^{(n-1)}(\Omega', \Lambda')a(\Omega'; \Phi(F)).\end{aligned}\]
Substituting this in to (\ref{eqn:B-using-lattice-count}) we have
\begin{equation}\label{eqn:B-after-j-gone} \begin{aligned} B(\Omega', \Lambda'; F|\widetilde{T}^{(n)}_j(p^2)) &= A(\Omega', \Lambda'; \Phi(F)| \widetilde{T}^{(n-1)}_j(p^2)) \\
&\qquad+A(\Omega^{(1)}, \Lambda; F|\widetilde{T}^{(n)}_j(p^2)) \\
&\qquad+ p^l \chi(p^{j-n}[\Omega^{(2)}:p\Lambda]) p^{E_j(\Omega^{(2)}, \Lambda')} \alpha^{(n-1)}_{j-1}(\Omega', \Lambda')a(\Omega'; \Phi(F)) \\
&\qquad+ p^{r+l} A(\Omega^{(3)}, \Lambda; F|\widetilde{T}^{(n)}_j(p^2)) \\
&\qquad+ p^l \sum_{\overline{u'}} A(\Omega^{(4)}(\overline{u'}), \Lambda; F|\widetilde{T}^{(n)}_j(p^2)). \end{aligned}\end{equation}
From the formulas
\[\begin{aligned} E^{(n)}_j(\Omega^{(1)}, \Lambda) &= 2k - j - n + E^{(n-1)}_{j-1}(\Omega', \Lambda'),\\ 
E^{(n)}_j(\Omega^{(2)}, \Lambda) &= -l - n + k + E^{(n-1)}_{j-1}(\Omega', \Lambda'),\\
E^{(n)}_j(\Omega^{(3)}, \Lambda) &= -r - l + n - j + E^{(n-1)}_{j-1}(\Omega', \Lambda'), \end{aligned} \]
and
\[\begin{aligned} \chi(p^{j-n}[\Omega^{(1)}:p\Lambda]) &= \chi(p^2)\chi(p^{j-n}[\Omega': p\Lambda']), \\
\chi(p^{j-n}[\Omega^{(2)}:p\Lambda]) &= \chi(p)\chi(p^{j-n}[\Omega': p\Lambda']), \\
\chi(p^{j-n}[\Omega^{(3)}:p\Lambda]) &= \chi(p^{j-n}[\Omega': \Lambda']),\end{aligned}\]
together with (\ref{eqn:iso-omega1}) and (\ref{eqn:iso-omega3}) we easily compute
\[\begin{aligned} A(\Omega^{(1)}, \Lambda; F|\widetilde{T}^{(n)}_j(p^2)) &= \chi(p^2)p^{2k-j-n}A(\Omega', \Lambda'; \Phi(F)|\widetilde{T}^{(n-1)}_{j-1}(p^2)), \\
 p^l \chi(p^{j-n}[\Omega^{(2)}:p\Lambda]) p^{E_j(\Omega^{(2)}, \Lambda)} \alpha^{(n-1)}_{j-1}(\Omega', \Lambda')a(\Omega'; \Phi(F))
 &= \chi(p)p^{k-n} A(\Omega', \Lambda'; \Phi(F)|\widetilde{T}^{(n-1)}_{j-1}(p^2)), \\
 p^{r+l} A(\Omega^{(3)}, \Lambda; F|\widetilde{T}^{(n)}_j(p^2)) &= p^{n-j} A(\Omega', \Lambda'; \Phi(F)|\widetilde{T}^{(n-1)}_{j-1}(p^2)). \end{aligned} \]
Substituting these in to (\ref{eqn:B-after-j-gone}) we obtain
\begin{equation}\label{eqn:B-after-j-1-gone} \begin{aligned} B(\Omega', \Lambda'; F|\widetilde{T}^{(n)}_j(p^2)) &= A(\Omega', \Lambda'; \Phi(F)| \widetilde{T}^{(n-1)}_j(p^2)) \\
&\qquad + (\chi(p^2)p^{2k-j-n} + \chi(p)p^{k-n} + p^{n-j})A(\Omega', \Lambda'; \Phi(F)|\widetilde{T}^{(n-1)}_{j-1}(p^2)) \\
&\qquad +p^l \sum_{\overline{u'}} A(\Omega^{(4)}(\overline{u'}), \Lambda; F|\widetilde{T}^{(n)}_j(p^2)). \end{aligned}\end{equation}
Finally, from 
\[\begin{aligned} E_j(\Omega^{(4)}(\overline{u'}), \Lambda) &= 2k-j-n-l - E_{j-2}(\Omega', \Lambda') \\
\chi(p^{j-n}[\Omega^{(4)}(\overline{u'}):p\Lambda]) &= \chi(p^2)\chi(p^{j-n-1}[\Omega':p\Lambda']),\end{aligned}\] 
and (\ref{eqn:iso-omega4}) we compute
\[\begin{aligned} & p^l \sum_{\overline{u'}} A(\Omega^{(4)}(\overline{u'}), \Lambda; F|\widetilde{T}^{(n)}_j(p^2))\\
&\qquad = \chi(p^2)p^{2k-j-n}(p^{2k-2j+1}-p^{2k-j-n})A(\Omega', \Lambda'; \Phi(F)|\widetilde{T}^{(n-1)}_{j-2}(p^2))\end{aligned} \]
Substituting this in to (\ref{eqn:B-after-j-1-gone}) we obtain
\begin{equation} \begin{aligned} B(\Omega', \Lambda'; \widetilde{T}^{(n)}_j(p^2)) &= A(\Omega', \Lambda'; \widetilde{T}^{(n-1)}_j(p^2)) \\
&\qquad+ (\chi(p^2)p^{2k-j-n} + \chi(p)p^{k-n} + p^{n-j})A(\Omega', \Lambda'; \widetilde{T}^{(n-1)}_{j-1}(p^2)) \\
&\qquad+ \chi(p^2)(p^{2k-2j+1}-p^{2k-j-n})A(\Omega', \Lambda'; \widetilde{T}^{(n-1)}_{j-2}(p^2))\end{aligned}\end{equation}
This is (\ref{eqn:suffices-to-prove}), so the proof is complete.
\end{proof}

From this it is straightforward to deduce Theorem \ref{thm:first-relation}:

\begin{proof}  [Proof of Theorem \ref{thm:first-relation} for $T_j^{(n)}(p^2)$]  Applying $\Phi$ to the definition (\ref{eqn:tilde-T-definition}) we have
\[\Phi(F|\widetilde{T}_j^{(n)}(p^2)) = p^{(n-j)(n-k+1)} \overline{\chi}(p^{n-j}) \sum_{t=0}^j \left(\frac{n-t}{j-t}\right)_p \Phi(F | T_t^{(n)}(p^2)).\]
Now it is clear from Proposition \ref{prop:tilde-intertwining-relations} that 
\[\Phi(F|T_t^{(n)}(p^2)) = \sum_{s=0}^t c_{t, s}^{(n-1)}\Phi(F)|T^{(n-1)}_s(p^2)\] 
for some complex numbers $c_{t, s}^{(n-1)}$.  Thus we can write
\begin{equation}\label{eqn:deduce-nontilde-LHS} \Phi(F|\widetilde{T}_j^{(n)}(p^2)) = p^{(n-j)(n-k+1)} \overline{\chi}(p^{n-j}) \sum_{t=0}^j \left(\frac{n-t}{j-t}\right)_p \sum_{s=0}^t c_{t, s}^{(n-1)}\Phi(F)|T^{(n-1)}_s(p^2).\end{equation}

On the other hand
\[\begin{aligned} \Phi(F|\widetilde{T}_j^{(n)}(p^2)) &= \Phi(F) | \widetilde{T}_j^{(n-1)}(p^2) + \widetilde{c}_{j, j-1}^{(n-1)}\Phi(F) | \widetilde{T}_j^{(n-1)}(p^2) \\
&\qquad + \widetilde{c}_{j, j-2}^{(n-1)}\Phi(F) | \widetilde{T}_j^{(n-1)}(p^2) \end{aligned}\]
which we can write as
\begin{equation}\label{eqn:deduce-nontilde-RHS}\begin{aligned}&\Phi(F|\widetilde{T}_j^{(n)}(p^2)) \\
&\qquad= p^{(n-1-j)(n-k)}\overline{\chi}(p^{n-1-j})\sum_{t=0}^j \left(\frac{n-1-t}{j-t}\right)_p \Phi(F)|T^{(n-1)}_t(p^2) \\
&\qquad\qquad + \widetilde{c}_{j, j-1}^{(n-1)}p^{(n-j)(n-k)}\overline{\chi}(p^{n-j})\sum_{t=0}^{j-1}\left(\frac{n-1-t}{j-1-t}\right)_p \Phi(F)|T_t^{(n-1)}(p^2) \\
&\qquad\qquad + \widetilde{c}_{j, j-2}^{(n-1)}p^{(n-j+1)(n-k)}\overline{\chi}(p^{n+1-j})\sum_{t=0}^{j-2} \left(\frac{n-1-t}{j-2-t}\right)_p \Phi(F)|T_t^{(n-1)}(p^2).\end{aligned} \end{equation}
Comparing the coefficient of $\Phi(f)|T_j^{(n-1)}(p^2)$ between (\ref{eqn:deduce-nontilde-LHS}) and (\ref{eqn:deduce-nontilde-RHS}) we have
\[p^{(n-j)(n-k+1)}\overline{\chi}(p^{n-j}) c_{j, j}^{(n-1)} = p^{(n-1-j)(n-k)}\overline{\chi}(p^{n-1-j})\]
from which we get $c_{j, j}^{(n-1)} = \chi(p)p^{j+k-2n}$.  Arguing similarly but with more tedious computation we compute the remaining coefficients and deduce Theorem \ref{thm:first-relation}. \end{proof}

\section{The intertwining relation for $\Phi$ and $T^{(n)}(p)$}\label{sctn:easy-relations}

We now describe how one can use a similar (but much easier) argument to that of \S\ref{sctn:first-relations} to derive the intertwining relation for the operator $T^{(n)}(p):=T^{(n)}(p, \chi)$ (dropping the character from the notation for this section only for ease of notation).  As before let $F \in \mathcal{M}_k^{(n)}(N, \chi)$ have Fourier expansion (\ref{eqn:lattice-fourier-expansion}).  Applying the Hecke operator $T^{(n)}(p)$ then the Siegel lowering operator $\Phi$ to we obtain
\begin{equation}\label{eqn:Tp-then-phi}\Phi(F | T^{(n)}(p))(Z') = \sum_{\Lambda'} \sum_{p\Lambda \subset \Omega \subset \frac{1}{p}\Lambda} A(\Omega, \Lambda; F|T^{(n)}(p))e\{\Lambda' Z'\};\end{equation}
and if we apply $\Phi$ first then $T^{(n-1)}(p)$ we obtain
\begin{equation}\label{eqn:phi-then-Tp}(\Phi(F) | T^{(n-1)}(p))(Z') = \sum_{\Lambda'} \sum_{p\Lambda' \subset \Omega' \subset \frac{1}{p}\Lambda'} A(\Omega', \Lambda'; \Phi(F)|T^{(n-1)}(p)) e\{\Lambda' Z'\},\end{equation}
and we must compare the Fourier coefficients in (\ref{eqn:Tp-then-phi}) and (\ref{eqn:phi-then-Tp}).  Fix an indexing lattice $\Lambda'$.  Let $\Omega'$ be a rank $n-1$ lattice, and define
\[B(\Omega', \Lambda'; F|{T}^{(n)}(p)) = \sum_{\Omega\text{ s.t.}\pi(\Omega) = \Omega'} A(\Omega, \Lambda; F|{T}^{(n)}(p)).\]
As in the proof of Proposition \ref{prop:tilde-intertwining-relations} we find that it suffices to show that
\begin{equation}\label{eqn:easy-suffices-to-prove} \begin{aligned} B(\Omega', \Lambda'; F|{T}^{(n)}(p)) &= (1 + \chi(p)p^{k-n}) A(\Omega', \Lambda'; \Phi(F)|{T}^{(n-1)}(p^2)) \end{aligned}\end{equation}
for each $p\Lambda' \subset \Omega' \subset \frac{1}{p}\Lambda'$.\\

\noindent It is again useful to classify all the lattices $\Omega$ which project on to a given $\Omega'$, and record some of the properties of such $\Omega$.  This is provided by the following two lemmas:

\begin{lemma}  There is a one-to-one correspondence between: \begin{enumerate}
\item  lattices $\Omega$ such that $p\Lambda \subset \Omega \subset \Lambda$ with $p$-type $(s, n-s)$,
\item  $s$-dimensional subspaces $\overline{\Delta}$ of $\Lambda/p\Lambda$. \end{enumerate} \end{lemma}

\begin{lemma}\label{lem:easy-lattice-counts}  Let $\Omega'$ be a lattice with $p\Lambda' \subset \Omega' \subset \Lambda'$ and $p$-type $(r, n-r-1)$.  Under the map $\pi : \Lambda \to \Lambda'$, the lattices that project on to $\Omega'$ are classified as follows: \renewcommand{\labelenumi}{(\Alph{enumi})}\begin{enumerate}
\item  one lattice with $p$-type $(r+1, n-r-1)$, which we will denote $\Omega^{(1)}$,
\item $p^r$ lattices with $p$-type $(r, n-r)$, which we will denote by $\Omega^{(2)}((\overline{\alpha_i})_{1 \leq i \leq s})$, where $\overline{\alpha_i} \in \mathbb{F}_p$.  \end{enumerate} \end{lemma}

The proofs are similar to (but easier than) Lemmas \ref{lattice-parameterisations} and \ref{lattice-counts}.  Then writing $\Omega^{(2)}$ for any $\Omega^{(2)}(\overline{\alpha_i})$ we compute, using the notation of Theorem \ref{action-of-Tp-on-fourier-exp}, 
\[\begin{aligned} E^{(n)}(\Omega^{(1)}, \Lambda) &= k-n + E^{(n-1)}(\Omega', \Lambda'), \\
E^{(n)}(\Omega^{(2)}, \Lambda) &= -r + E^{(n-1)}(\Omega', \Lambda'), \end{aligned}\]
and
\[\begin{aligned} \chi([\Omega^{(1)}:p\Lambda]) &= \chi(p)\chi([\Omega':p\Lambda']), \\
\chi([\Omega^{(2)}:p\Lambda]) &= \chi([\Omega':p\Lambda']), \end{aligned}\]
so that
\begin{equation}\label{eqn:easy-relate-up-to-down}\begin{aligned} A(\Omega^{(1)}, \Lambda; F|T^{(n)}(p)) &= \chi(p)p^{k-n}A(\Omega', \Lambda'; \Phi(F)|T^{(n-1)}(p)) \\
A(\Omega^{(2)}, \Lambda; F|T^{(n)}(p) &= p^{-r} A(\Omega', \Lambda'; \Phi(F) | T^{(n-1)}(p)).\end{aligned}\end{equation}
But by Lemma \ref{lem:easy-lattice-counts} we have
\[B(\Omega', \Lambda'; F|T^{(n)}(p)) = A(\Omega^{(1)}, \Lambda; F|T^{(n)}(p)) + p^r A(\Omega^{(2)}, \Lambda; F|T^{(n)}(p)).\]
Substituting (\ref{eqn:easy-relate-up-to-down}) in to this we obtain (\ref{eqn:easy-suffices-to-prove}).  This proves the intertwining relation for $T^{(n)}(p)$ stated in Theorem \ref{thm:first-relation}.

\section{Review of the Satake compactification}\label{satake-compactification-general}

Let $\Gamma^{(n)}$ be a congruence subgroup of $\Sp_{2n}(\Q)$, so that $\Gamma^{(n)}$ acts on $\mathfrak{H}_n$ by \eqref{eqn:gsp4-action-on-hn}, the resulting quotient space $\Gamma^{(n)} \backslash \mathfrak{H}_n$ is a complex analytic space of dimension $n(n+1)/2$.  There are various approaches to compactifying this space in the literature but the simplest and most important for the classical theory of Siegel modular forms is the \textit{Satake compactification}.  We will briefly review this construction; our account is based on \cite{PoorYuen2013}, in which a very explicit description of the cuspidal structure in the case of paramodular level is also given.  In the following section we will provide a similar explicit description for level $\Gamma_0^{(n)}(N)$ when $N$ is squarefree.\\  

\noindent Let $\C^{2n \times n}_{\text{rank }n} \subset \C^{2n \times n}$ be the subset of rank $n$ matrices.  Let
\[\Gr_\C(2n, n) = \C^{2n \times n}_{\text{rank }n} / \GL_n(\C)\]
be the Grassmannian of rank $n$ subspaces of $\C^{2n}$, and write $\left[\begin{smallmatrix} M \\ N\end{smallmatrix}\right] \in \Gr_\C(2n, n)$ for the class of $\left(\begin{smallmatrix} M \\ N \end{smallmatrix}\right) \in \C^{2n \times n}_{\text{rank }n}$.  Consider the subspace of isotropic subspaces
\[\Gr_\C^{\text{iso}}(2n, n) = \left\{ \left[\begin{matrix} M \\ N \end{matrix}\right] \in \Gr_\C(2n, n);\: \left(\begin{matrix} {}^t M & {}^tN \end{matrix}\right) \left(\begin{matrix} 0_n & -1_n \\ 1_n & 0_n\end{matrix}\right) \left(\begin{matrix} M \\ N \end{matrix}\right) = 0\right\}.\]
We shall endow $\Gr_\C^{\text{iso}}(2n, n)$ with the complex structure it naturally inherits from these definitions.  We let $\Sp_{2n}(\C)$ act on $\Gr_\C^{\text{iso}}(2n, n)$ via matrix multiplication on the left.\\

Consider the upper half space $\mathfrak{H}_r$ for $0 \leq r \leq n$, with the convention $\mathfrak{H}_0 = \{\infty\}$ is a singleton.  Let $j_{r, n}: \mathfrak{H}_r \hookrightarrow \Gr_\C(2n, n)$ be given by
\[j_{r, n}(Z) = \left[\begin{matrix} 1_n \\ \left(\begin{smallmatrix} Z^{-1} & 0 \\ 0 & 0 \end{smallmatrix}\right) \end{matrix}\right];\]
when $r=0$ we mean that the bottom $n \times n$ block is $0_n$.  One easily sees that $j_{r, n}(\mathfrak{H}_r) \subset \overline{j_{n, n}(\mathfrak{H}_n)}$ (the closure taking place inside $\Gr_\C^{\text{iso}}(2n, n)$).  For $0 \leq r \leq n$ consider the orbit $\Sp_{2n}(\Q)j_{r, n}(\mathfrak{H}_r)$; note that when $r=n$ this is just $\mathfrak{H}_n$, but for $r<n$ it is strictly larger than $\mathfrak{H}_r$.  Now define a subspace $\mathfrak{H}_n^*$ of $\Gr_\C^{\text{iso}}(2n, n)$ by
\[\mathfrak{H}_n^* = \bigsqcup_{r=0}^n \Sp_{2n}(\Q) j_{r, n}(\mathfrak{H}_{r}).\]
Then $\mathfrak{H}_n^*$ is naturally equipped with an action of $\Gamma^{(n)}$, and the \textit{Satake compactification} of $\Gamma^{(n)} \backslash \mathfrak{H}_n$ is simply the quotient $\Gamma^{(n)} \backslash \mathfrak{H}_n^*$.  Now $\Gamma^{(n)} \backslash \mathfrak{H}_n^*$, being a subquotient of $\Gr_\C(2n, n)$, comes equipped with a natural topology, under which it becomes a compact Hausdorff space.  We note that
\begin{equation}\label{eqn:r-cusp-representatives} \begin{aligned} \Sp_{2n}(\Q) j_{r, n}(\mathfrak{H}_r) = \bigsqcup_i \Gamma^{(n)} \gamma_i j_{r, n}(\mathfrak{H}_r) \end{aligned}\end{equation}
where the $\gamma_i$ are a system of representatives for 
\[\Gamma^{(n)} \backslash \Sp_{2n}(\Q) / \Stab_{\Sp_{2n}(\Q)}(j_{r, n}(\mathfrak{H}_r))\]  
One can explicitly compute that this stabiliser is equal to $P_{n, r}(\Q)$, where $P_{n, r} \subset \Sp_{2n}$ is the parabolic subgroup defined in (\ref{eqn:Pnr-definition}).  Recall also the surjection $\omega_{n, r}$ from (\ref{eqn:omeganr-definition}); this is split by the map $\xi_{r, n} : \Sp_{2r}(\Q) \to P_{n, r}(\Q)$ defined by
\[\xi_{r, n}\left(\begin{matrix} A_{11} & B_{11} \\ C_{11} & D_{11} \end{matrix}\right) = \left(\begin{matrix} A_{11} & 0 & B_{11}& 0 \\ 0 & 1_{n-r} & 0 & 0_{n-r} \\ C_{11} & 0 & D_{11} & 0 \\ 0 & 0_{n-r} & 0 & 1_{n-r} \end{matrix}\right).\]
Now consider an individual $\Gamma^{(n)} \gamma j_{r, n}(\mathfrak{H}_r)$ in (\ref{eqn:r-cusp-representatives}).  Let 
\[\Gamma^{(r)}_\gamma := \omega_{n, r} (\gamma^{-1} \Gamma^{(n)} \gamma \cap P_{n, r}(\Q)).\]  
The map defined on $\mathfrak{H}_r$ by
\[Z \mapsto \Gamma^{(n)} \backslash \Gamma^{(n)} \gamma j_{r, n}(Z)\]
induces an isomorphism
\[\Gamma^{(r)}_\gamma \backslash \mathfrak{H}_r \to \Gamma^{(n)} \backslash \Gamma^{(n)} \gamma j_{n, r}(\mathfrak{H}_r) \subset \Gamma^{(n)} \backslash \mathfrak{H}_n^*.\]
The space $\Gamma^{(r)}_\gamma \backslash \mathfrak{H}_r$ is therefore embedded inside $\Gamma^{(n)} \backslash \mathfrak{H}_n^*$.  This is a space of dimension $r(r+1)/2$, and we shall (temporarily) refer to this as the $r$-cusp of $\Gamma^{(n)} \backslash \mathfrak{H}_n^*$ associated to $\gamma$.\\

Now an $r$-cusp $\Gamma^{(r)}_\gamma \backslash \mathfrak{H}_r$ is the quotient of an upper half plane by a congruence subgroup, and we could therefore form the compactification as above.  Abstractly this would involve forming the space $\mathfrak{H}_r^* = \bigsqcup_{s=0}^r \Sp_{2r}(\Q)j_{s, r}(\mathfrak{H}_{s})$, and writing
\[\Sp_{2r}(\Q)j_{s, r}(\mathfrak{H}_{s}) = \bigsqcup_i \Gamma^{(r)}_\gamma \rho_i j_{s, r}(\mathfrak{H}_s),\]
where the $\rho_i$ are a system of representatives for 
\[\Gamma^{(r)}_\gamma \backslash \Sp_{2r}(\Q) / P_{r, s}(\Q).\]  
However we want to construct this compactification not with a new incarnation but rather in the already-carnate space $\Gamma^{(n)} \backslash \mathfrak{H}_n^*$.  We therefore extend the embedding of $\Gamma^{(r)}_\gamma \backslash \mathfrak{H}_r$ to an embedding of $\Gamma^{(r)}_\gamma \bs \mathfrak{H}_r^*$ as follows:  given $0 \leq s \leq r$, $Z \in \mathfrak{H}_s$, $\rho \in \Sp_{2r}(\Q)$, and $\gamma \in \Sp_{2n}(\Q)$ consider the map
\[\Gamma^{(r)}_\gamma \rho j_{s, r}(Z) \mapsto \Gamma^{(n)} \gamma \xi_{r, n}(\rho) j_{s, n}(Z).\]
This induces a well-defined isomorphism
\[\Gamma^{(r)}_\gamma \backslash \Gamma^{(r)}_\gamma \rho j_{s, r}(\mathfrak{H}_s) \to \Gamma^{(n)} \backslash \Gamma^{(n)} \gamma \xi_{r, n}(\rho) j_{s, n}(\mathfrak{H}_s).\]
Varying $s$ and $\rho$ we obtain an embedding $\Gamma^{(r)}_\gamma \backslash \mathfrak{H}_r^* \hookrightarrow \Gamma^{(n)} \backslash \mathfrak{H}_n^*$ (in fact, the image of $\Gamma^{(r)}_\gamma \bs \mathfrak{H}_r^*$ under this embedding is simply the closure of $\Gamma^{(r)}_\gamma \bs \mathfrak{H}_r$ in $\Gamma^{(n)} \backslash \mathfrak{H}_n^*$).  We shall replace our earlier convention and now call $\Gamma^{(r)}_\gamma \bs \mathfrak{H}_r^*$ (viewed inside $\Gamma^{(n)} \backslash \mathfrak{H}_n^*$) the $r$-cusp of $\Gamma^{(n)} \bs \mathfrak{H}_N^*$ associated to $\gamma$.

\begin{remark}\label{rmk:boundary-component-depends-on-rep}  The arithmetic subgroup $\Gamma^{(r)}_\gamma$, and hence the structure of the cusp $\Gamma^{(r)}_\gamma \backslash \mathfrak{H}_r^*$, depends on the choice of representative $\gamma$ for $\Gamma^{(n)}\gamma P_{n, r}(\Q)$.  More precisely, it is invariant under left multiplication by $\Gamma^{(n)}$, but changes by a conjugation if we right multiply by some element of $P_{n, r}(\Q)$.  Similarly, one may work with instead with the double coset space $\Gamma^{(n)} \backslash \GSp_{2n}(\Q) / P_{n, r}^*(\Q)$ (which is in bijection with $\Gamma^{(n)} \backslash \Sp_{2n}(\Q) / P_{n, r}(\Q)$) where $P_{n, r}^*$ is the parabolic subgroup of $\GSp_{2n}$ which contains $P_{n, r}$, and a similar statement holds.  For the remainder of this section and for \S\ref{sctn:satake-compactification-description} we are only interested in properties of the double cosets so this remark is unimportant.  However, this technicality will become important from \S\ref{sctn:intertwining-relations-any-cusp} onwards.\end{remark}

\noindent We now record some observations regarding cusp crossings: let $\gamma$ represent an $r$-cusp of $\Gamma^{(n)} \backslash \mathfrak{H}_n^*$ and $\rho$ represent an $s$-cusp of $\Gamma^{(r)}_\gamma \backslash \mathfrak{H}_r^*$.  By the above embedding, the latter may be thought of as an $s$-cusp of $\Gamma^{(n)} \backslash \mathfrak{H}_n^*$; explicitly this is the $s$-cusp given by the double coset $\Gamma^{(n)} \gamma \xi_{r, n}(\rho) P_{n, s}(\Q)$.  Then:
\begin{itemize} \item  if this same coset can be realised with two inequivalent $\gamma$ and $\gamma'$ (i.e. the double cosets $\Gamma^{(n)}\gamma P_{n, r}(\Q)$ and $\Gamma^{(n)}\gamma' P_{n, r}(\Q)$ are different), then the two distinct $r$-cusps corresponding to $\gamma$ and $\gamma'$ intersect at this $s$-cusp,
\item  if this same coset can be obtained with the same (or just equivalent) $\gamma$ but inequivalent $\rho$ and $\rho'$ (i.e. $\Gamma_\gamma^{(r)} \rho P_{r, s}(\Q)$ and $\Gamma_\gamma^{(r)} \rho' P_{r, s}(\Q)$ are different) then the $r$-cusp corresponding to $\gamma$ self-intersects at this $s$-cusp.  \end{itemize}

\section{The Satake compactification of $\Gamma_0^{(n)}(N) \backslash \mathfrak{H}_n$}\label{sctn:satake-compactification-description}

In this section we will provide an explicit description of the cuspidal configuration of $\Gamma_0^{(n)}(N) \backslash \mathfrak{H}_n^*$, where $N$ is square-free.  It is well-known that for $n=1$ that the $0$-cusps are in bijection with positive divisors of $N$.  For $n=2$ one must consider not only $1$- and $0$-cusps, but also how the former may cross at the latter.  An account of this is given in \cite{BoechererIbukiyama2012}.  Motivated by this we will proceed analogously for general $n$.  \\

\noindent Recall from \S\ref{satake-compactification-general} that the $r$-cusps of $\Gamma_0^{(n)}(N) \backslash \mathfrak{H}_n^*$ correspond bijectively to 
\[\Gamma_0^{(n)}(N) \backslash \Sp_{2n}(\Q) / P_{n, r}(\Q).\]  
We begin by describing representatives for this.  For each $1 \leq r \leq n$ and each divisor $l$ of $N$ fix a matrix $\gamma^{(r)}(l) \in \Sp_{2r}(\Z)$ satisfying
\begin{equation}\label{eqn:def-of-basic-cusp-reps}\gamma^{(r)}(l) \equiv \begin{cases} \begin{pmatrix} 0_r & -1_r \\ 1_r & 0_r \end{pmatrix} \bmod l^2, \\ \begin{pmatrix} 1_r & 0_r \\ 0_r & 1_r \end{pmatrix} \bmod (N/l)^2. \end{cases}\end{equation}
This is possible since, for all $M \in \Z_{\geq 1}$, the reduction modulo $M$ map $\Sp_{2r}(\Z) \to \Sp_{2r}(\Z/M\Z)$ is surjective.  Next, given a sequence of positive integers $l_1,...,l_{n-r}$, assumed to be pairwise coprime and each a divisor of $N$, define
\begin{equation} \label{eqn:standard-cusp-reps-def}\begin{aligned} \gamma_r^{(n)}(l_{n-r},...,l_1) &= \gamma^{(n)}(l_{n-r}) \xi_{n-1, n}(\gamma^{(n-1)}(l_{n-r-1})) \hdots \xi_{r+1, n}(\gamma^{(r+1)}(l_1)) \end{aligned}\end{equation}
Set $l_0 = N/(l_{n-r} \hdots l_1)$.  To explain the ordering of the indices, first note that
\[\begin{aligned} \gamma_r^{(n)}(l_{n-r},...,l_1) \equiv \begin{pmatrix} 0_{r+i} & & -1_{r+i} & \\ & 1_{n-r-i} & & 0_{n-r-i} \\ 1_{r+i} & & 0_{r+i} & \\ & 0_{n-r-i} & & 1_{n-r-i} \end{pmatrix} \bmod l_i^2\end{aligned}\]
for $n-r \geq i \geq 1$, and $\gamma_r^{(n)}(l_{n-r},...,l_1) \equiv 1_{2n} \bmod l_0^2$.  On the other hand, write any element $\gamma$ of $\Sp_{2n}(\Q)$ as $\left(\begin{smallmatrix} A & B \\ C & D \end{smallmatrix}\right)$, and $C$ in turn as $\left(\begin{smallmatrix} C_{11} & C_{12} \\ C_{21} & C_{22} \end{smallmatrix}\right)$ with $C_{22}$ size $n-r$.  Then the under the left action of $\Gamma_0^{(n)}(N)$ and the right action of $P_{n, r}(\Q)$ we see that $\rk_p(C_{22})$, the rank of $C_{22}$ modulo $p$, for $p \mid N$ is invariant.  Going back to $\gamma = \gamma_r^{(n)}(l_{n-r},...,l_1)$ with $l_{n-r}, ...,l_{1}$ pairwise coprime divisors of $N$, we see that $\{p \mid l_i\} = \{p;\: \rk_p(C_{22}) = i\}$.  With our definition this holds with $i=0$ as well.

\begin{lemma}\label{cusp-representatives}  Continue with the above notation.  Then as $(l_{n-r},...,l_1)$ varies over all tuples of pairwise coprime positive divisors of $N$, the $\gamma_r^{(n)}(l_{n-r},...,l_1)$ describe a complete system of coset representatives for
\[\Gamma_0^{(n)}(N) \backslash \Sp_{2n}(\Q) / P_{n, r}(\Q).\]\end{lemma}
\begin{proof}  By the discussion preceding the statement of the lemma we see that the $\gamma_r^{(n)}(l_{n-r},...,l_1)$ are inequivalent for distinct tuples $(l_{n-r},...,l_1)$.  One can argue further form these rank observations to see that the $\gamma_r^{(n)}(l_{n-r},...,l_1)$ actually form a complete set of representatives.  Alternatively this follows since they agree in number with those of \cite{BoechererSchulze-Pillot1991} Lemma 8.1.  \end{proof}

Henceforth we shall identify a cusp of $\Gamma_0^{(n)}(N) \backslash \mathfrak{H}_n^*$ with the corresponding tuple $(l_{n-r},...,l_1)$ of pairwise coprime positive divisors of $N$.  With $\gamma = \gamma(l_{n-r},...,l_1)$ one sees that
\[\Gamma_\gamma^{(r)} = \Gamma^{(r)}_0(l_0, l_{n-r}\hdots l_1),\]
where
\[\Gamma^{(r)}_0(N_1, N_2) = \left\{\begin{pmatrix} A & B \\ C & D \end{pmatrix};\: C \equiv 0 \bmod N_1;\: B \equiv 0 \bmod N_2\right\}.\]

\begin{remark}\label{rmk:boundary-could-be-gamma0N}  The group $\Gamma^{(r)}_0(N_1, N_2)$ is conjugate to the group $\Gamma^{(r)}_0(N_1N_2)$, so the space of modular forms on the boundary components is isomorphic to the space of modular forms on $\Gamma_0^{(r)}(N)$.  This is related to \eqref{eqn:lowering-dependency-on-rep} and Remark \ref{rmk:boundary-component-depends-on-rep}.  In fact, it is not difficult to see that one can also choose the representatives so that one manifestly has boundary components of the form $\Gamma_0^{(r)}(N)$.  The representatives we have chosen are convenient for the present computations; we will work with a slight modification of them in \S\ref{sctn:intertwining-relations-any-cusp} and \S\ref{action-of-hecke-on-eis} which will be well-suited to studying modular forms on the boundary components.\end{remark}

We now describe the intersections between these boundary components.  Of course, in contrast to the issues raised in Remark \ref{rmk:boundary-could-be-gamma0N}, this is purely a question about the double cosets and the result of this computation does not depend on the choice of representatives we have made.

\begin{theorem}\label{thm:description-of-boundary}  Let $n$ be a positive integer, let $N$ a square-free positive integer, and let $\Gamma_0^{(n)}(N) \backslash \mathfrak{H}_n^*$ the Satake compactification of $\Gamma_0^{(n)}(N) \backslash \mathfrak{H}_n$.
\begin{enumerate} \item  Let $(l_{n-r},...,l_1)$ be an $r$-cusp represented by $\gamma$ as above, and let $0 \leq s \leq r$.  Consider two $s$-cusps on the Satake compactification $\Gamma_\gamma^{(r)} \backslash \mathfrak{H}_r^*$ of the boundary component corresponding to $(l_{n-r},...,l_1)$.  If these two $s$-cusps are equal when viewed as $s$-cusps of $\Gamma_0^{(n)}(N) \backslash \mathfrak{H}_n^*$, then they are also equal when viewed as $s$-cusps of $\Gamma_0^{(r)}(N) \backslash \mathfrak{H}_r^*$.  In other words, no $r$-cusp can self-intersect at an $s$-cusp.
\item  Let $(l_{n-s},...,l_1)$ be an $s$-cusp, where $0 \leq s < n-1$.  Then the $(s+1)$-cusps on which $(l_{n-s},...,l_1)$ lies are precisely those of the form 
\[\left(l_{n-s}c_{n-s-1}, \frac{l_{n-s-1}}{c_{n-s-1}}c_{n-s-2}, \frac{l_{n-s-2}}{c_{n-s-2}} c_{n-s-3},...,\frac{l_2}{c_2} c_1\right),\]
where, for $1 \leq i \leq n-s-1$, $c_i \mid l_i$. \end{enumerate}
\end{theorem}

\begin{remark}  Let $0 \leq s < r < n$.  Part 2 of Theorem \ref{thm:description-of-boundary} can be applied inductively to describe which $r$-cusps an arbitrary $s$-cusp lies on.  Alternatively, enough ingredients will be given in the proof of Theorem \ref{thm:description-of-boundary} to describe this in general, although we omit it since it is notationally cumbersome.\end{remark}

Before proving Theorem \ref{thm:description-of-boundary} let us demonstrate the consistency of the numbers in it, since it may not be immediately obvious that this is the case.  We will count $s$-cusps with the multiplicity: More precisely, we count each $s$-cusps once for every $(s+1)$-cusps on which it appears.  Let $t$ denote the number of prime divisors of the squarefree integer $N$.  On the one hand the number of $(s+1)$-cusps of $\Gamma_0^{(n)}(N) \backslash \mathfrak{H}_n^*$ is the number of tuples $(l_{n-(s+1)}, ..., l_1)$ of pairwise coprime positive divisors of $N$, of which there are $(n-(s+1)+1)^t=(n-s)^t$; on each of these cusps the number of $s$-cusps is $((s+1)-s+1)^t = 2^t$, so the number of $s$-cusps with multiplicity is $2^t(n-s)^t$.\\

On the other hand, suppose we fix an $s$-cusp $(l_{n-s}',...,l_1')$.   Let us write $\epsilon_{i}$ for the number of prime divisors of $l_i'$.  Part 2 of Theorem \ref{thm:description-of-boundary} tells us that the number of $(s+1)$-cusps on which $(l_{n-s}',...,l_1')$ lies is $2^{\epsilon_{n-s-1}} \hdots 2^{\epsilon_2} 2^{\epsilon_1}$.  Write also $\delta_i = \sum_{j=1}^i \epsilon_j$ for the number of prime divisors of $l_i' \hdots l_1'$, so that $\epsilon_i = \delta_i - \delta_{i-1}$ for $i>1$.  Then the number of $s$-cusps counted with multiplicity according to the number of $(s+1)$-cusps on which they appear is
\[\sum_{\delta_{n-s} = 0}^t {t \choose \delta_{n-s}} \sum_{\delta_{n-s-1}=0}^{\delta_{n-s}} {\delta_{n-s} \choose \delta_{n-s-1}} \hdots \sum_{\delta_1=0}^{\delta_2} {\delta_2 \choose \delta_1} 2^{\delta_{n-s-1} - \delta_{n-s-2}} \hdots 2^{\delta_2 -\delta_1} 2^{\delta_1} = 2^t(n-s)^t,\]
by repeatedly applying the binomial theorem.

\begin{proof}  [Proof of Theorem \ref{thm:description-of-boundary}]  First note that if $(l_{n-r},...,l_1)$ is an $r$-cusp of $\Gamma_0^{(n)}(N) \backslash \mathfrak{H}_n^*$ and $(m_{r-s},...,m_1)$ is an $s$-cusp on this $r$-cusp then viewed inside $\Gamma_0^{(n)}(N) \backslash \mathfrak{H}_n^*$ this $s$-cusp is represented by the matrix
\[\begin{aligned}&\gamma_r^{(n)}(l_{n-r},...,l_1)\xi_{r, n}(\gamma_s^{(r)}(m_{r-s},...,m_1)) \\
&\qquad= \gamma^{(n)}(l_{n-r}) \xi_{n-1, n}(\gamma^{(n-1)}(l_{n-r-1})) \hdots \xi_{r+1, n}(\gamma^{(r+1)}(l_1)) \\
&\qquad\qquad \times \xi_{r, n}(\gamma^{(r)}(m_{r-s})) \xi_{r-1, r}(\gamma^{(r-1)}(m_{r-s-1})) \hdots \xi_{s+1, r}(\gamma^{(s+1)}(m_1)).\end{aligned}\]
This is of course not one of our representatives.  To determine this as an $s$-cusp of $\Gamma_0^{(n)}(N) \backslash \mathfrak{H}_n^*$ is to determine which coset it is in in the space $\Gamma_0^{(n)}(N) \backslash \Sp_{2n}(\Q) / P_{n, s}(\Q)$, which is simply to determine the rank of the $C_{22}$ block (of size $n-s$) of the above matrix modulo $p$ for each $p \mid N$.  We write $l_0 = N/(l_{n-r} \hdots l_1)$ and $m_0 = N/(m_{s-r} \hdots m_1)$.  By multiplying out in the expression for $\gamma_r^{(n)}(l_{n-r},...,l_1)\xi_{r, n}(\gamma_s^{(r)}(m_{r-s},...,m_1))$ one sees that
\begin{itemize}
\item  if $p \mid m_0$ and $p \mid l_0$ then $\rk_p(C_{22}) = 0$,
\item  if $p \mid m_0$ and $p \mid l_i$ where $n-r \geq i \geq 1$ then $\rk_p(C_{22}) = r-s+i$,
\item  if $p \mid m_j$ where $r-s \geq j \geq 1$ and $p \mid l_0$ then $\rk_p(C_{22}) = j$,
\item  if $p \mid m_j$ where $r-s \geq j \geq 1$ and $p \mid l_i$ where $n-r \geq i \geq 1$ then $\rk_p(C_{22}) = (r-s+i) -j$. 
\end{itemize}
This is enough to deduce Part 1.  Indeed, let $(m_{r-s}',...,m_1')$ be another $s$-cusp on $(l_{n-r},...,l_1)$, and let $C_{22}'$ be the corresponding block of size $(n-s)$.  We assume that this $s$-cusp when viewed inside $\Gamma_0^{(n)}(N) \backslash \mathfrak{H}_n^*$ is the same as the one coming from $(m_{r-s},...,m_1)$; equivalently $\rk_p(C_{22}) = \rk_p(C_{22}')$ for all $p \mid N$.  We claim that this implies $(m_{r-s},...,m_1) = (m_{r-s}',...,m_1')$.  Define $m_0' = N/(m_{r-s}' \hdots m_1')$.  We will prove that $\{p \mid m_i\}$ and $\{p \mid m_i'\}$ are the same; this is sufficient because everything is squarefree.  Take a divisor $p$ of $N$, and assume first that $p \mid l_0$.  From the above criteria we have under this assumption that, for $r-s \geq j \geq 0$,
\[p \mid m_j \iff \rk_p(C_{22}) = j \iff \rk_p(C_{22}') = j \iff p \mid m_j'.\]
Now assume that $p \mid l_i$ where $n-r \geq i \geq 1$.  Again from the above criteria we have, for $r-s \geq j \geq 0$,
\[p \mid m_j \iff \rk_p(C_{22}) = r-s-j \iff \rk_p(C_{22}') = r-s-j \iff p \mid m_j'.\]
Since every $p \mid N$ divides some $l_i$, this proves Part 1.  In fact, we see that if $(l_{n-r},...,l_1)$ is an $r$-cusp, and $(m_{r-s},...,m_1)$ is an $s$-cusp on it, then the $s$-cusp when viewed inside $\Gamma_0^{(n)}(N) \backslash \mathfrak{H}_n^*$ is $(l'_{n-s}, ..., l_1')$ where, for $r-s < i \leq n-s$,
\[l'_i = (m_{r-s}, l_i) (m_{r-s-1}, l_{i-1}) \hdots (m_1, l_{i-(r-s-1)}) (m_0, l_{i-(r-s)}),\]
and for $1 \leq i \leq r-s$,
\[l'_i = (m_{r-s}, l_i) (m_{r-s-1}, l_{i-1}) \hdots (m_{r-s-(i-1)}, l_1) (m_i,l_0).\]
In the above formulas, if we refer to $(l_a, m_b)$ where either $l_a$ or $m_b$ is not defined (e.g. $a \leq -1$ or $a \geq n-r+1$) then we understand that $(l_a, m_b)$ should be omitted from the product.\\

In order to prove Part 2 we must start with an $(s+1)$-cusp, say $(d_{n-s-1}, ,...,d_1)$, and exhibit an $s$-cusp $m_1$ on this which is equal to $(l_{n-s},...,l_1)$, when viewed inside $\Gamma_0^{(n)}(N) \backslash \mathfrak{H}_n$.  Following the recipe above, where we are taking $r = s+1$, we see that we must exhibit $(m_1, m_0)$ with $m_1 m_0 = N$ such that
\begin{align*} l_{n-s} &= &&(m_0, d_{n-s-1}) && \\
l_{n-s-1} &= &&(m_0, d_{n-s-2}) &&(m_1, d_{n-s-1}) \\
l_{n-s-2} &= &&(m_0, d_{n-s-3}) &&(m_1, d_{n-s-2}) \\
& &&\qquad\vdots && \\
l_3 &= &&(m_0, d_2) &&(m_1, d_3) \\
l_2 &= &&(m_0, d_1) &&(m_1, d_2) \\
l_1 &= && (m_1, d_0) && (m_1, d_1)\\
l_0 &= && &&(m_0, d_0).\end{align*}
If $d_{n-s-1} = l_{n-s} c_{n-s-1}$ and $d_i = (l_{i+1}/c_{i+i})c_i$ for $n-s-2 \geq i \geq 1$ as in the statement of Part 2 then we take
\[\begin{aligned} m_0 &= l_{n-s} \cdot \frac{l_{n-s-1}}{c_{n-s-1}} \cdot \frac{l_{n-s-2}}{c_{n-s-2}} \cdots \frac{l_3}{c_3} \cdot \frac{l_2}{c_2} \cdot 1 \cdot l_0, \\
m_1 &= 1 \cdot c_{n-s-1} \cdot c_{n-s-2} \cdots c_3 \cdot c_2 \cdot l_1 \cdot 1; \end{aligned}\]
this is written so as to emphasize which primes of $l_i$ are in $m_0$ and $m_1$ respectively.  To finish it remains to show, given $(l_{n-s},...,l_1)$, that if we have an $(s+1)$-cusp $(d_{n-s-1},...,d_1)$ which satisfies the above system equations (for some $m_1$) then it must be of the form stated in Part 2 of the Theorem.  Now examining the equations for $l_{n-s}$ and $l_{n-s-1}$ we see that $d_{n-s-1}$ must be a multiple of $l_{n-s}$ which divides $l_{n-s}l_{n-s-1}$, so $d_{n-s-1} = l_{n-s} c_{n-s-1}$ for some $c_{n-s-1} \mid l_{n-s-1}$.  Next examining the equations for $l_{n-s-1}$ and $l_{n-s-2}$ we see that $d_{n-s-2}$ must be a multiple of $l_{n-s-1}/c_{n-s-1}$ which divides $(l_{n-s-1}/c_{n-s-1})l_{n-s-2}$, so $d_{n-s-2} = (l_{n-s-1}/c_{n-s-1})c_{n-s-2}$ for some $c_{n-s-2} \mid l_{n-s-2}$.  This pattern continues all the way up to $d_1$, and we see that it is necessary that $(d_{n-s-1},...,d_1)$ has the form stated in Part 2 of the Theorem.  Since we've already seen that this is sufficient we are done. \end{proof}

\section{Intertwining relations at arbitrary cusps for squarefree level}\label{sctn:intertwining-relations-any-cusp}

We continue with the imposition that $N$ be squarefree.  In this section we will prove Theorem \ref{thm:any-cusp-relation}.  In the following section we will show how the can be used to obtain information on the action of Hecke operators on Klingen--Eisenstein series.\\

Write 
\[\kappa^{(n)}(l) = \begin{pmatrix} 1_n & 0_n \\ 0_n & l 1_n \end{pmatrix} \gamma^{(n)}(l)\]
where $\gamma^{(n)}(l)$ is as in (\ref{eqn:def-of-basic-cusp-reps}), so that
\[\kappa^{(n)}(l) \equiv \begin{cases}\begin{pmatrix} 0_n & -1_n \\ l 1_n & 0_n \end{pmatrix} \bmod l^2 \\ \begin{pmatrix} 1_n & 0_n \\ 0_n & l 1_n \end{pmatrix} \bmod (N/l)^2.\end{cases}\]
As $l$ varies over all positive divisors of $N$ the $\kappa(l)$ represent the double coset space 
\[\Gamma^{(n)}_0(N) \backslash \GSp_{2n}(\Q) / P_{n, n-1}^*(\Q),\]
where $P_{n, r}^*$ is the parabolic subgroup of $\GSp_{2n}$ which contains $P_{n, r}$ (i.e. the similtudes preserving the same flag).  The inclusion induces a bijection
\[\Gamma^{(n)}_0(N) \backslash \Sp_{2n}(\Q) / P_{n, r}(\Q) \simeq \Gamma^{(n)}_0(N) \backslash \GSp_{2n}(\Q) / P_{n, r}^*(\Q),\]
so that the $\kappa^{(n)}(l)$ are in bijection with the $(n-1)$-cusps of $\Gamma_0^{(n)}(N) \backslash \mathfrak{H}_n^*$.  An easy computation shows that
\[\kappa(l)^{-1} \Gamma_0^{(n)}(N) \kappa(l) = \Gamma_0^{(n)}(N),\]
and that the map $f \mapsto f |_k \kappa(l)$ defines an isomorphism $\mathcal{M}_k^{(n)}(N, \chi) \to \mathcal{M}_k^{(n)}(N, \overline{\chi_l} \chi_{N/l})$.  For a $l$ positive divisor of $N$ we write $\Phi_l$ for the operator defined by
\[\Phi_l(F) = \Phi(F |_k \kappa(l)),\]
so $\Phi_1 = \Phi$.

\begin{remark}\label{rmk:similtude-lowering-not-well-defined}  As in \eqref{eqn:lowering-dependency-on-rep}, this definition depends on the choice of representative.  More precisely, if $\gamma' \in \Gamma^{(n)}$ and $\delta \in P_{n, n-1}^*(\Q)$,
\[\Phi(F |_k \gamma' \kappa \delta) = D_{22}^{-k} \chi(\gamma') \mu_n(\delta)^{(n-r)k/2} \Phi(F | \kappa) | \omega_{n, n-1}(\delta).\]\end{remark}

As a map of vector space we have $\Phi_l : \mathcal{M}_k^{(n)}(N, \chi) \to \mathcal{M}_k^{(n-1)}(N, \overline{\chi}_l \chi_{N/l})$.  In terms of the Hecke module structure at primes not dividing the level we have the following:

\begin{lemma}\label{commutative-diagrams-lemma}  Let $n$ and $k$ be positive integers, $N$ a squarefree positive integer, and $p$ a prime not dividing $N$.  For $l \mid N$ let $\kappa(l)$ be as above.  Then we have the following commutative diagrams:
\[\begin{CD} \mathcal{M}_k^{(n)}(N, \chi) @>T^{(n)}(p, \chi)>> \mathcal{M}_k^{(n)}(N, \chi)\\
@V |_k \kappa(l) VV @VV |_k \kappa(l) V\\
\mathcal{M}_k^{(n)}(N, \overline{\chi_l} \chi_{N/l}) @>>\chi_l(p^n)T^{(n)}(p, \overline{\chi_l}\chi_{N/l})> \mathcal{M}_k^{(n)}(N, \overline{\chi_l}\chi_{N/l}), \end{CD}\]
and for $1 \leq j \leq n$
\[\begin{CD} \mathcal{M}_k^{(n)}(N, \chi) @>T_j^{(n)}(p^2, \chi)>> \mathcal{M}_k^{(n)}(N, \chi)\\
@V |_k \kappa(l) VV @VV |_k \kappa(l) V\\
\mathcal{M}_k^{(n)}(N, \overline{\chi_l} \chi_{N/l}) @>>\chi_l(p^{2n}) T_j^{(n)}(p^2, \overline{\chi_l}\chi_{N/l})> \mathcal{M}_k^{(n)}(N, \overline{\chi_l}\chi_{N/l}).\end{CD}\]\end{lemma}
\begin{proof}  We shall show commutativity of the first diagram using an argument based on \cite{Miyake2006} Theorem 4.5.5; the second will yield to similar reasoning.  \\

Write $\alpha = \left(\begin{smallmatrix} 1_n & \\ & p1_n \end{smallmatrix}\right)$, so that $T(p, \chi)$ is given by the double coset $\Gamma_0^{(n)}(N) \alpha \Gamma_0^{(n)}(N)$.  Note that $\mu(\alpha) = p$.  Write 
\begin{equation}\label{eqn:Tp-rep-names} \Gamma_0^{(n)}(N) \alpha \Gamma_0^{(n)}(N) = \bigsqcup_v \Gamma_0^{(n)}(N) \alpha_v.\end{equation}  
Since $\kappa(l)^{-1} \Gamma_0^{(n)}(N) \kappa(l) = \Gamma_0^{(n)}(N)$ we also have 
\begin{equation}\label{eqn:Tp-rep-names-conj}\Gamma_0^{(n)}(N) \alpha \Gamma_0^{(n)}(N) = \bigsqcup_v \Gamma_0^{(n)}(N) \kappa(l)^{-1} \alpha_v \kappa(l).\end{equation}  
Now take $F \in \mathcal{M}_k^{(n)}(N, \overline{\chi}_l \chi_{N/l})$, then
\[F | \kappa(l)^{-1} | T(p, \chi) | \kappa(l) = p^{\frac{nk}{2} - \frac{n(n+1)}{2}}\sum_v \overline{\chi}(\alpha_v) F | \kappa(l)^{-1} \alpha_v \kappa(l)\]
where we have chosen the decomposition (\ref{eqn:Tp-rep-names}) for our definition of $T(p, \chi)$.  Writing $\alpha_v = \left(\begin{smallmatrix} A_v & B_v \\ C_v & D_v \end{smallmatrix}\right)$ we have $\overline{\chi}(\alpha_v) = \chi(\det(A_v))$, so
\begin{equation}\label{top-path-of-cd} F | \kappa(l)^{-1} | T(p, \chi) | \kappa(l) = p^{\frac{nk}{2} - \frac{n(n+1)}{2}}\sum_v \chi(\det(A_v)) F | \kappa(l)^{-1} \alpha_v \kappa(l).\end{equation}
On the other hand, 
\[F | T(p, \overline{\chi}_l \chi_{N/l}) = p^{\frac{nk}{2} - \frac{n(n+1)}{2}}\sum_v \overline{\chi}(\kappa(l)^{-1} \alpha_v \kappa(l)) F | \kappa(l)^{-1} \alpha_v \kappa(l)\]
where we have chosen the decomposition (\ref{eqn:Tp-rep-names-conj}) for our definition of $T(p, \overline{\chi}_l \chi_{N/l})$.  This is seen to be the same as
\[ F | T(p, \overline{\chi}_l \chi_{N/l}) = p^{\frac{nk}{2} - \frac{n(n+1)}{2}} \sum_v \overline{\chi}_l(\det(D_v)) \chi_{N/l}(\det(A_v)) F |\kappa(l)^{-1} \alpha_v \kappa(l).\]
Now $\det(A)\det(D) \equiv \det(\alpha) \equiv p^n \bmod N$, so $\overline{\chi}_l(\det(D_v)) = \overline{\chi}_l(p^n) \chi_l(\det(A))$, hence
\begin{equation}\label{bottom-path-of-cd} F | T(p, \overline{\chi}_l \chi_{N/l}) = \overline{\chi}_l(p^n) p^{\frac{nk}{2} - \frac{n(n+1)}{2}} \sum_v \chi(\det(A_v)) F |\kappa(l)^{-1} \alpha_v \kappa(l)\end{equation}
Comparing (\ref{top-path-of-cd}) and (\ref{bottom-path-of-cd}) we see that if we multiply the latter by $\chi_l(p^{n})$ then we obtain the former; whence we obtain the stated commutative diagram.  \end{proof}

\begin{proof}[Proof of Theorem \ref{thm:any-cusp-relation}]  This follows immediately from Theorem \ref{thm:first-relation} and Lemma \ref{commutative-diagrams-lemma}.\end{proof}

\section{Action of Hecke operators on Klingen--Eisenstein series}\label{action-of-hecke-on-eis}

We now define Klingen--Eisenstein series, which is most conveniently done iteratively: rather going directly from an $r$-cusp all the way to $\Gamma_0^{(n)}(N) \backslash \mathfrak{H}_n^*$, we proceed via a sequence of $r$-cusps.  As usual, let $N$ be a squarefree positive integer and $\chi$ a Dirichlet character modulo $N$.  Let $l_1 \mid N$ represent an $(n-1)$-cusp, and set $l_0 = N/l_1$.  We take $F \in \mathcal{M}_k^{(n-1)}(N, \overline{\chi_{l_1}} \chi_{l_0})$.  Fix the representative $\kappa^{(n)}(l_1)$ from \S\ref{sctn:intertwining-relations-any-cusp} for $l_1$.  We define
\[E^{(n)}_{l_1}(Z; F) = \mu_n(\kappa(l_1))^{-nk/2} \sum_{M} \overline{\chi}(\kappa(l_1) M) j(M, Z)^{-k} F(\pi (M\langle Z \rangle)),\]
where $M$ varies over a system of representatives of 
\[(\kappa(l_1)^{-1} \Gamma_0^{(n)}(N) \kappa(l_1) \cap P_{n, n-1}(\Q)) \backslash \kappa(l_1)^{-1} \Gamma^{(n)}_0(N).\]  
Since $\kappa(l_1)^{-1} \Gamma_0^{(n)}(N) \kappa(l_1) = \Gamma_0^{(n)}(N)$ we can simply say that $M$ varies over a system of representatives of 
\[\Gamma_0^{(n)}(N) \cap P_{n, n-1}(\Q) \backslash \kappa(l_1)^{-1}\Gamma_0^{(n)}(N).\]
One easily checks that $E_{l_1}(\cdot; F)$ is well-defined provided that
\[\overline{\chi}(\kappa(l_1) \delta \kappa(l_1)^{-1}) D_{22}^{-k} = 1\text{, for all } \delta \in \kappa(l_1)^{-1} \Gamma_0^{(n)}(N) \kappa(l_1) \cap P_{n, n-1}(\Q),\]
which is equivalent to
\[\chi(-1) = (-1)^k,\]
an assumption which we have tacitly presumed throughout in light of Remark \ref{rmk:character-parity}.  The series defining $E^{(n)}_{l_1}(F) := E^{(n)}_{l_1}(\cdot; F)$ converges absolutely provided that $k > 2n$, so under this assumption we have $E_{l_1}(F) \in \mathcal{M}_k^{(n)}(N, \chi)$.  For $l_1 \mid N$ and $F \in \mathcal{M}_k^{(n-1)}(N, \overline{\chi_{l_1}} \chi_{l_0})$ we have
\begin{equation}\label{eqn:level-N-eis-section-to-phi}\Phi_{l_1}(E_{l_1}(\cdot; F)) = F,\end{equation}
as one easily sees with an application of dominated convergence (using the assumption $k>2n$).\\ 

Note that the form $F \in \mathcal{M}_k^{(n-1)}(N, \overline{\chi}_{l_1} \chi_{l_0})$ we lift need not be an a cusp form, but if it is then we easily prove the following:

\begin{lemma}\label{lem:klingen-drop-one-degree}  Let $N \in \Z_{\geq 1}$ be squarefree, $l_1 \mid N$ represent an $(n-1)$-cusp of $\Gamma_0^{(n)} \backslash \mathfrak{H}_n^*$, and set $l_0 = N/l_1$.  Let $F \in \mathcal{S}_k^{(n-1)}(N, \overline{\chi_{l_1}}\chi_{l_0})$ be an eigenfunction of $T^{(n-1)}(p, \overline{\chi_{l_1}}\chi_{l_0})$ with eigenvalue $\lambda^{(n-1)}(p, \overline{\chi_{l_1}}\chi_{l_0})$.  Then 
\[E_{l_1}(\cdot; F) | T^{(n)}(p, \chi) = \lambda^{(n)}(p, \chi)E_{l_1}(\cdot; F),\] 
where
\[\lambda^{(n)}(p, \chi) = \left(\chi_{l_1}(p^n) + {\chi}_{l_1}(p^{n-1})\chi_{l_0}(p)p^{k-n}\right)\lambda^{(n-1)}(p, \overline{\chi_{l_1}}\chi_{l_0}).\]\end{lemma}  
\begin{proof}  First note that the Eisenstein subspace is invariant under the action of Hecke operators outside the primes dividing the level.  This is very well-known, and is easily be proved using the (obvious) fact that the Hecke operators preserve the subspace of cusp forms, and the fact that Hecke operators at $p \nmid N$ on $\mathcal{M}_k^{(n)}(N, \chi)$ are normal with respect to the Petersson inner product (\cite{Andrianov2009} Lemma 4.6).  Thus $E_{l_1}(\cdot; f)|T^{(n)}(p, \chi)$ is an Eisenstein series.  Let $l_1'$ be any divisor of $N$, and set $l_0' = N/l_1'$.  Then
\[\begin{aligned} &\Phi_{l_1'} (E_{l_1}(\cdot; F) | T^{(n)}(p, \chi)) \\
&\qquad = \left(\chi_{l_1'}(p^n) + {\chi}_{l_1'}(p^{n-1})\chi_{l_0'}(p)p^{k-n}\right) \Phi_{l_1'} (E_{l_1}(\cdot; F)) | T^{(n-1)}(p, \overline{\chi_{l_1'}}\chi_{l_0'}).\end{aligned}\]
If $l_1'=l_1$ this becomes
\begin{equation}\label{eqn:value-of-klingen-Tp-at-d}\begin{aligned} &\Phi_{l_1}(E_{l_1}(\cdot; F)|T^{(n)}(p, \chi)) \\
&\qquad = \left(\chi_{l_1}(p^n) + {\chi}_{l_1}(p^{n-1})\chi_{l_0}(p)p^{k-n}\right)\lambda^{(n-1)}(p, \overline{\chi_{l_1}} \chi_{l_0})\Phi_{l_1}(E_{l_1}(\cdot; F)).\end{aligned}\end{equation}
If $l_1' \neq l_1$ we instead get
\begin{equation}\label{eqn:value-of-klingen-Tp-at-d'}\Phi_{l_1'}(E_{l_1}(\cdot; F)|T^{(n)}(p, \chi)) = 0.\end{equation}
Now consider the function
\[E_{l_1}(\cdot; F)|T^{(n)}(p, \chi) - (\chi_{l_1}(p^n) + {\chi}_{l_1}(p^{n-1})\chi_{l_0}(p)p^{k-n})E_{l_1}(\cdot; F) \in \mathcal{M}_k^{(n)}(N, \chi).\]
By (\ref{eqn:value-of-klingen-Tp-at-d}) and (\ref{eqn:value-of-klingen-Tp-at-d'}) this vanishes at all $(n-1)$-cusps, so is a cusp form.  On the other hand, by the discussion at the beginning of the proof it is an Eisenstein series.  Thus it must be equal to zero.\end{proof}

\noindent The same argument also proves the following:

\begin{lemma}\label{lem:klingen-drop-one-degree-j}  Let $F \in \mathcal{S}_k^{(n-1)}(N, \overline{\chi_{l_1}} \chi_{l_0})$ be an eigenfunction all of $T_j^{(n-1)}(p^2, \overline{\chi_{l_1}} \chi_{l_0})$, $0 \leq j \leq n-1$, with eigenvalues $\lambda_j^{(n-1)}(p^2, \overline{\chi_{l_1}} \chi_{l_0})$.  Then
\[E_{l_1}(\cdot; F) | T_j^{(n)}(p^2, \chi) = \lambda^{(n)}(p^2, \chi) E_{l_1}(\cdot; F),\]
where
\[\begin{aligned} \lambda^{(n)}(p^2, \chi) &= \chi_{l_1}(p^{2n})\left[c_{j, j}^{(n-1)}(\overline{\chi_{l_1}} \chi_{l_0})\lambda^{(n-1)}_j(p^2, \overline{\chi_{l_1}}\chi_{l_0}) \right.\\
&\left.\qquad+ c_{j, j-1}^{(n-1)}(\overline{\chi_{l_1}}\chi_{l_0})\lambda^{(n-1)}_{j-1}(p^2, \overline{\chi_{l_1}} \chi_{l_0}) \right.\\
&\left.\qquad+ c_{j, j-2}^{(n-1)}(\overline{\chi_{l_1}}\chi_{l_0})\lambda_{j-2}^{(n-1)}(p^2, \overline{\chi_{l_1}} \chi_{l_0})\right],\end{aligned}\]
where $c_{j, j}^{(n-1)}, c_{j, j-1}^{(n-1)}, c_{j, j-2}^{(n-2)}$ are as in Theorem \ref{thm:first-relation}.\end{lemma}

With a little more book-keeping we can generalise Lemmas \ref{lem:klingen-drop-one-degree} and \ref{lem:klingen-drop-one-degree-j} to all Klingen--Eisenstein series.  Write $\mathcal{M}_k^{(n, n)}(N, \chi) = \mathcal{S}_k^{(n)}(N, \chi) \subset \mathcal{M}_k^{(n)}(N, \chi)$ for the subspace of cusp forms.  It makes sense to define the orthogonal complement $\mathcal{N}_k^{(n)}(N, \chi)$ of $\mathcal{M}_k^{(n, n)}(N, \chi)$ with respect to the Petersson inner product \eqref{eqn:petersson-def}.  Define an operator
\[\widetilde{\Phi} : \mathcal{N}_k^{(n)}(N, \chi) \to \oplus_{l_1 \mid N} \mathcal{M}_k^{(n-1)}(N, \overline{\chi}_{l_1} \chi_{l_0})\]
by
\[F \mapsto (\Phi_{l_1}(F))_{l_1 \mid N}.\]
This map is not surjective, since the vectors in the image must agree on lower dimensional intersections.  However, in large enough weights, $\widetilde{\Phi}$ surjects on to the subspace cut out by this condition, as we shall see in a moment.  First, we define some subspaces $\mathcal{M}_k^{(n, i)}(N, \chi) \subset \mathcal{N}_k^{(n)}(N, \chi)$ for $0 \leq i < n$ by induction on $n$.  There is nothing to do for $n=1$: for any character $\psi$ modulo $N$, $\mathcal{M}_k^{(1, 0)}(N, \psi) = \mathcal{N}_k^{(1)}(N, \psi)$ is the usual space of degree one Eisenstein series.  For $n>1$ and $0 \leq i < n$, we define 
\[\mathcal{M}_k^{(n, i)}(N, \chi) = \widetilde{\Phi}^{-1} \left(\bigoplus_{l_1 \mid N} \mathcal{M}_k^{(n-1, i)}(N, \overline{\chi}_{l_1} \chi_{l_0})\right).\]
Then $\mathcal{M}_k^{(n, i)}(N, \chi)$ is a linear subspace, and $\Sigma_{i=0}^n \mathcal{M}_k^{(n, i)}(N, \chi)$ is in fact direct.  By double induction (increasing on $n$, decreasing on $i$) one sees from the normality of the Hecke operators with respect to the inner product (\ref{eqn:petersson-def}) and Theorem \ref{thm:any-cusp-relation} that the Hecke algebra $\mathcal{H}_p^{(n)}$ (when $p \nmid N$) preserves the decomposition $\bigoplus_i \mathcal{M}_k^{(n, i)}(N, \chi)$.\\

In order to produce some elements of $\mathcal{M}_k^{(n, i)}(N, \chi)$ we will use Eisenstein series.  At the same time this will show $\widetilde{\Phi}$ is surjective, i.e.
\begin{equation}\label{eqn:stratification-of-mkn}\mathcal{M}_k^{(n)}(N, \chi) = \bigoplus_{i=0}^n \mathcal{M}_k^{(n, i)}(N, \chi).\end{equation}
We work iteratively: let $(l_{n-r},...,l_1)$ be a sequence of pairwise coprime divisors of $N$ corresponding to an $r$-cusp, and define
\[\Phi_{(l_{n-r},...,l_1)} = \Phi_{l_1} \circ \Phi_{l_2} \cdots \circ \Phi_{l_{n-r}}.\]
In the other direction, let $F \in \mathcal{S}_k^{(r)}(N, \overline{\chi}_{l_{n-r} \hdots l_1} \chi_{l_0})$, and define
\[E_{(l_{n-r},...,l_1)}(F) = E_{l_{n-r}} \circ E_{l_{n-r-1}} \circ \cdots \circ E_{l_1}(F).\]
The proof of (\ref{eqn:stratification-of-mkn}) follows easily by induction once we know that $E_{(l_{n-r},...,l_1)}(F) \in \mathcal{M}_k^{(n, n-r)}$.  This latter fact follows from a somewhat technical computation for which we refer to \cite{Satake1957}, especially (2.12) and the discussion preceding it.  See also \cite{Harris1984} Corollary 2.4.6, which includes a detailed proof of this decomposition but is slightly removed from our context since modular forms are identified with sections of automorphic vector bundles.\\

We can now iterate the idea of Lemmas \ref{lem:klingen-drop-one-degree} and \ref{lem:klingen-drop-one-degree-j} to handle lifts of cusp forms of any degree $0 \leq r < n$:

\begin{theorem}\label{thm:action-of-hecke-on-klingen}  Let $n$ be a positive integer, $0 \leq r < n$, $N$ be a squarefree positive integer, and $(l_{n-r},...,l_1)$ correspond to an $r$-cusp of $\Gamma_0^{(n)}(N) \backslash \mathfrak{H}_n^*$.  Let $F \in \mathcal{S}_k^{(r)}(N, \overline{\chi}_{l_{n-r}\hdots l_1}\chi_{l_0})$, where $k>n+r+1$.
\begin{enumerate}
\item  Assume that $F$ is an eigenfunction of $T^{(r)}(p, \overline{\chi}_{l_{n-r}\hdots l_1}\chi_{l_0})$ with eigenvalue $\lambda^{(r)}(p, \overline{\chi}_{l_{n-r}\hdots l_1}\chi_{l_0})$.  Then 
\[E_{(l_{n-r},...,l_1)}(F) | T^{(n)}(p, \chi) = \lambda^{(n)}(p, \chi)E_{(l_{n-r},...,l_1)}(F),\] 
where
\[\lambda^{(n)}(p, \chi) = \lambda^{(r)}(p, \overline{\chi}_{l_{n-r}\hdots l_1} \chi_{l_0}) \prod_{t=r+1}^n \chi_{l_{t-r}}(p^t) c^{(t)}(\overline{\chi}_{l_{n-r} \hdots l_{t-r}} \chi_{l_{t-r-1} \hdots l_0}).\]
\item  Assume that $F$ is an eigenfunction of each $T_j^{(r)}(p^2, \overline{\chi}_{l_{n-r} \hdots l_1} \chi_{l_0})$ with eigenvalues $\lambda_j^{(r)}(p^2, \overline{\chi}_{l_{n-r} \hdots l_1} \chi_{l_0})$.  Then 
\[E_{(l_{n-r},...,l_1)}(F) | T_j^{(n)}(p^2, \chi) = \lambda_j^{(n)}(p^2, \chi)E_{(l_{n-r},...,l_1)}(F),\] 
where $\lambda_j^{(n)}(p^2, \chi)$ is given by the following recursive procedure: Define
\[\begin{aligned}\theta(m, i, \psi, l_t) &= \psi_{l_{t}}(p^{2m}) \left[c_{i, i}^{(m-1)}(\overline{\psi}_{l_{t}}\psi_{N/l_t}) \lambda_i^{(m-1)}(p^2, \overline{\psi}_{l_{t}} \psi_{N/l_{t}}) \right.\\
&\left.\qquad\qquad + c_{i, i-1}^{(m-1)}(\overline{\psi}_{l_{t}}\psi_{N/l_t}) \lambda_{i-1}^{(m-1)}(p^2, \overline{\psi}_{l_{t}} \psi_{N/l_{t}}) \right.\\
&\left.\qquad\qquad + c_{i, i-2}^{(m-1)}(\overline{\psi}_{l_{t}}\psi_{N/l_t}) \lambda_{i-2}^{(m-1)}(p^2, \overline{\psi}_{l_{t}} \psi_{N/l_{t}})\right],\end{aligned}\]
where $c_{i, i}^{(m-1)}(\cdot)$, $c_{i, i-1}^{(m-1)}(\cdot)$, and $c_{i, i-2}^{(m-1)}(\cdot)$ are given by Theorem \ref{thm:first-relation}, we have the convention that $c^{(s)}_{j, k} = 0$ if $k < 0$ or $k > s$, and the quantities $\lambda$ on the right hand side are currently treated as formal variables.  Then the eigenvalue can be computed by applying the above formula
\[\lambda^{(n)}_j(p^2, \chi) = \theta(n, j, \chi, l_{n-r}).\]
If $r=n-1$ then we substitute in the eigenvalues $\lambda_i^{(n-1)}(p^2, \overline{\chi}_{l_{n-r}} \chi_{l_{n-r-1} \cdots l_0})$ (for $i-2 \leq j \leq i$) of the underlying cusp form and terminate; otherwise we compute these quantities by again applying the formula
\[\lambda_i^{(n-1)}(p^2, \overline{\chi}_{l_{n-r}} \chi_{l_{n-r-1} \cdots l_0}) = \theta(n-1, i, \overline{\chi}_{l_{n-r}} \chi_{l_{n-r-1} \cdots l_0}, l_{n-r-1}).\]
This procedure terminates once we have applied the formula $n-r$ times, and gives an expression for $\lambda_j^{(n)}(p^2, \chi)$ in terms of the eigenvalues of $F$.
\end{enumerate} 
\end{theorem}
\begin{proof}  We prove Part 1 by induction on $n$, the proof of Part 2 follows by the same argument.  When $n=r+1$ this is Lemma \ref{lem:klingen-drop-one-degree}, so the base case is done.  In general, consider the function
\begin{equation}\label{eqn:in-right-stratum}E_{(l_{n-r},...,l_1)}(F)| T^{(n)}(p, \chi) - \lambda^{(n)}(p, \chi) E_{(l_{n-r},...,l_1)}(F) \in \mathcal{M}_k^{(n, n-r)}(N, \chi).\end{equation}
Then
\[\begin{aligned} &\Phi_{l_{n-r}}(E_{(l_{n-r},...,l_1)}(F)| T^{(n)}(p, \chi) - \lambda^{(n)}(p, \chi) E_{(l_{n-r},...,l_1)}(F)) \\
&\qquad = \chi_{l_{n-r}}(p^n)c^{(n)}(\overline{\chi}_{l_{n-r}}\chi_{l_{n-r-1}\hdots l_0}) T^{(n-1)}(p, \overline{\chi}_{l_{n-r}}\chi_{l_{n-r-1} \hdots l_1}) E_{(l_{n-r-1},...,l_1)}(F) \\
&\qquad\qquad - \lambda^{(n)}(p, \chi) E_{(l_{n-r-1},...,l_1)}(F) \\
&\qquad = \chi_{l_{n-r}}(p^n)c^{(n)}(\overline{\chi}_{l_{n-r}}\chi_{l_{n-r-1}\hdots l_0}) \left[T^{(n-1)}(p, \overline{\chi}_{l_{n-r}}\chi_{l_{n-r-1} \hdots l_1}) E_{(l_{n-r-1},...,l_1)}(F) \right.\\
&\left.\qquad\qquad - \lambda^{(n-1)}(p, \overline{\chi}_{l_{n-r}} \chi_{l_{n-r-1} \hdots l_0})E_{(l_{n-r-1},...,l_1)}\right].\end{aligned}\]
By induction hypothesis this is zero.  On the other hand it is clear that $\Phi_{(l_{n-r}',...,l_1')}(E_{(l_{n-r},...,l_1)}(F)) = 0$ if $(l_{n-r}',...,l_1') \neq (l_{n-r},...,l_1)$.  Thus
\[\Phi_{(l_{n-r}',...,l_1')}(E_{(l_{n-r},...,l_1)}(F)| T^{(n)}(p, \chi) - \lambda^{(n)}(p, \chi) E_{(l_{n-r},...,l_1)}(F)) = 0,\]
for all $(l_{n-r}',...,l_1')$.  But the containment in (\ref{eqn:in-right-stratum}) tells us that
\[E_{(l_{n-r},...,l_1)}(F)| T^{(n)}(p, \chi) - \lambda^{(n)}(p, \chi) E_{(l_{n-r},...,l_1)}(F)\]
is determined by its value on all $r$-cusps.  Since we have shown it vanishes at all of these, it must be zero, so we obtain the statement of the theorem.\end{proof}


\bibliographystyle{plain}
\bibliography{refs}

\end{document}